\documentclass[12pt]{amsart}
\usepackage{txfonts}      
\usepackage{amssymb}
\usepackage{eucal}
\usepackage{amsmath}
\usepackage{amscd}
\usepackage{xcolor}
\usepackage{multicol}
\usepackage[all]{xy}           
\usepackage{graphicx}
\usepackage{color}
\usepackage{colordvi}
\usepackage{xspace}
\usepackage{tikz}
\usepackage{makecell}
\usepackage{appendix}
\usepackage{enumerate,enumitem}
\usepackage{amsthm}
\usepackage[thicklines]{cancel}
\usepackage[compress]{cite}
\usepackage{ifpdf}
\ifpdf
\usepackage[colorlinks,final,backref=page,hyperindex]{hyperref}
\else
\usepackage[colorlinks,final,backref=page,hyperindex,hypertex]{hyperref}
\fi

\usepackage[active]{srcltx} 
\newtheorem{thm}{Theorem}[section]
\newtheorem{lem}[thm]{Lemma}
\newtheorem{cor}[thm]{Corollary}
\newtheorem{pro}[thm]{Proposition}
\theoremstyle{definition}
\newtheorem{defi}[thm]{Definition}
\newtheorem{ex}[thm]{Example}
\newtheorem{rmk}[thm]{Remark}

\title[Leibniz-dendriform bialgebras and relative Rota-Baxter operators] {Leibniz-dendriform bialgebras and relative Rota-Baxter operators}

\author{Qinxiu Sun}
\address{Department of Mathematics, Zhejiang University of Science and Technology, Hangzhou, 310023} \email{qxsun@126.com}

\author{Shuangjian Guo$^*$}
\address{ School of Mathematics and Statistics, Guizhou University of Finance and Economics, Guiyang, 550025}
         \email{shuangjianguo@126.com}


\begin{document}
\begin{abstract}
In this paper, we introduce the notion of Leibniz-dendriform bialgebras and establish their equivalence with phase 
 spaces and matched pairs of Leibniz algebras. The study of the coboundary case leads naturally to
  the Leibniz-dendriform Yang-Baxter equation (LD-YBE). 
We prove that skew-symmetric solutions of the LD-YBE give rise to coboundary Leibniz-dendriform bialgebras.
 Furthermore, we demonstrate that solutions not necessarily skew-symmetric can also induce such bialgebras.
This observation motivates the introduction of quasi-triangular and factorizable 
Leibniz-dendriform bialgebras. In particular, we show that solutions of the LD-YBE 
with invariant symmetric parts yield quasi-triangular Leibniz-dendriform bialgebras. 
Such solutions are also interpreted as relative Rota-Baxter operators with weights.
Finally, we establish a one-to-one correspondence between quadratic Rota-Baxter Leibniz-dendriform algebras
 and factorizable Leibniz-dendriform bialgebras.
 
\end{abstract}

\maketitle

\vspace{-1.2cm}

\tableofcontents

\vspace{-1.2cm}

\allowdisplaybreaks

\section{Introduction}
Leibniz algebras  were first introduced and investigated by Blokh \cite{B65, B71} under the name of $D$-algebras, 
emphasizing their close relationship with derivations. Later, during his study of Lie algebra homology, 
Loday observed that the antisymmetry of the product was not necessary to establish the derivation property on chains. 
This led him to rediscover and systematically define the notion of right (or, equivalently, left) Leibniz algebras \cite{L93}. 
Subsequently, Loday and Pirashvili laid the groundwork for their representation theory by introducing 
corresponding module structures and developing an associated cohomology theory \cite{LP93}.
Since then, Leibniz algebras have found diverse applications in mathematics and physics. 
From an operadic perspective, the operad governing Leibniz algebras is the duplicator of the Lie operad \cite{52}. 
Notably, they serve as the underlying algebraic structures of embedding tensors and play a significant role in higher gauge theories \cite{51,50}. 
Just as a Rota-Baxter operator on a Lie algebra induces a pre-Lie algebra structure, so does
 a (relative) Rota-Baxter operator on a Leibniz algebra yield a pre-Leibniz algebra \cite{Da}, 
 which has already introduced in \cite{28} under the name of Leibniz-dendriform algebra.
 In a recent study, the authors \cite{34} explored symplectic structures and
  phase spaces on Leibniz algebras via the framework of Leibniz-dendriform algebras. 
  Their work establishes a one-to-one correspondence between 
  symplectic Leibniz algebras and quadratic Leibniz-dendriform algebras.
 
A bialgebraic structure integrates an algebra with a coalgebra via specific compatibility conditions
 that govern their multiplication and comultiplication.
 In the early 1980s, Drinfeld established the theoretical groundwork for Lie bialgebras in \cite{8},
  highlighting two pivotal aspects: 
 their close connection to the classical Yang-Baxter equation and their central role in the infinitesimalization 
 of quantum groups. Subsequent developments expanded directly on this groundwork: 
 V. Zhelyabin introduced the notion of associative $D$-bialgebras in \cite{32,33},
  while Aguiar developed the antisymmetric infinitesimal bialgebra in \cite{1}.
   This antisymmetric structure bears dual significance: 
   it is equivalent to the double construction of Frobenius algebras \cite{4} and 
   can be interpreted as a Manin triple of associative algebras endowed with 
   a non-degenerate symmetric invariant bilinear form. Notably, this approach based on 
   Manin triples for studying bialgebraic structures has been successfully extended to various other algebraic structures, 
 such as pre-Lie algebras \cite{3}, Poisson algebras \cite{26}, Leibniz algebras \cite{28}, 
 perm algebras \cite{19}, Novikov algebras \cite{16}, pre-Novikov algebras \cite{22} and anti-dendriform algebras \cite{030}.

Within the framework of Lie bialgebras, coboundary Lie bialgebras with quasi-triangular Lie bialgebras as 
their prominent subclass play a foundational role in mathematical physics. Meanwhile, factorizable 
Lie bialgebras \cite{022} serve as a crucial bridge in this context:
they connect classical $r$-matrices to certain types of factorization problems, and their usefulness extends to
 diverse applications in integrable systems \cite{102,103}. In recent studies,
  these key results concerning factorizable and quasi-triangular structures have moved 
  beyond the setting of Lie bialgebras, achieving successful generalizations to other algebraic frameworks.
  Specifically, they have been adapted and generalized to antisymmetric infinitesimal bialgebras \cite{029}, 
  pre-Lie bialgebras \cite{033}, and Leibniz bialgebras \cite{07}.

It is natural to investigate bialgebra theory for Leibniz-dendriform algebras,
 which serves as the primary motivation for this work. Specifically, we introduce the notion of a Leibniz-dendriform bialgebra.
The study of coboundary Leibniz-dendriform bialgebras leads to the Leibniz-dendriform Yang-Baxter equation (LD-YBE),
  where any skew-symmetric solution gives rise to a Leibniz-dendriform bialgebra. More importantly,
   we examine how solutions without skew-symmetry can induce such bialgebras. In particular,
    we prove that solutions of the LD-YBE whose symmetric parts are invariant yield quasi-triangular Leibniz-dendriform bialgebras.
    Furthermore, we consider factorizable Leibniz-dendriform bialgebras as a special subclass of
     quasi-triangular Leibniz-dendriform bialgebras.
     We show that the double space of any Leibniz-dendriform bialgebra naturally carries a factorizable structure.
  Finally, we characterize solutions of the LD-YBE with invariant symmetric parts in terms of
  relative Rota-Baxter operators on Leibniz-dendriform algebras.

 The paper is organized as follows. In Section 2, we review fundamental concepts and results related
 to Leibniz algebras and Leibniz-dendriform algebras. In particular, we study the representations and
 matched pairs of Leibniz-dendriform algebras.
  Section 3 builds upon these foundations to develop a bialgebra theory for Leibniz-dendriform algebras.
  By examining the coboundary case, we introduce the Leibniz-dendriform Yang-Baxter equation (LD-YBE),
  whose skew-symmetric solutions yield Leibniz-dendriform bialgebras. We further introduce $\mathcal{O}$-operators
   on Leibniz-dendriform algebras and Leibniz-quadri-algebras, utilizing both structures to generate 
   skew-symmetric solutions of the LD-YBE. In Section 4, we explore quasi-triangular and factorizable Leibniz-dendriform bialgebras.
  We prove that the double of any Leibniz-dendriform bialgebra naturally carries a factorizable structure.
  In Section 5, we introduce quadratic Rota-Baxter Leibniz-dendriform algebras and
  establish their correspondence with factorizable Leibniz-dendriform bialgebras.

{\bf Notations.} Throughout the paper, $k$ is a field.  All vector spaces and algebras are over $k$.
 All algebras are finite-dimensional, although many results still hold in the infinite-dimensional case.
 Let $V$ be a vector space with a binary operation $\ast$. Define linear maps
$L_{\ast}, R_{\ast}:V\rightarrow \hbox{End}(V)$ by
 $L_{\ast}(a)b:=a\ast b, \  \ R_{\ast}(a)b:=b\ast a$
 and $\tau:V\otimes V\longrightarrow V\otimes V ,~\tau(a\otimes b)=b\otimes a$ for all$~a, b\in V$.
Assume that
 $r=\sum\limits_{i}a_i\otimes b_i \in V\otimes V$. Put
 \begin{small}
\begin{align*}
r_{12}\ast r_{13}:=\sum_{i,j}a_i\ast a_j\otimes b_i\otimes b_j,\;r_{23}\ast r_{12}:=\sum_{i,j}a_j\otimes a_i\ast b_j\otimes b_i,\;
r_{31}\ast r_{23}:=\sum_{i,j}b_i\otimes a_j\otimes a_i\ast b_j,\\
r_{21}\ast r_{13}:=\sum_{i,j}b_i\ast a_j\otimes a_i\otimes b_j,\;
r_{13}\ast r_{12}:=\sum_{i,j}a_i\ast a_j\otimes b_j\otimes b_i,\;
r_{13}\ast r_{23}:=\sum_{i,j}a_i\otimes a_j\otimes b_i\ast b_j,\\
r_{13}\ast r_{32}:=\sum_{i,j}a_i\otimes b_j\otimes b_i\ast a_j,\;
r_{23}\ast r_{21}:=\sum_{i,j}b_j\otimes a_i\ast a_j\otimes b_i,\;
r_{21}\ast r_{31}:=\sum_{i,j}b_i\ast b_j\otimes a_i\otimes a_j,\\
r_{23}\ast r_{13}:=\sum_{i,j}a_i \otimes a_j \otimes b_i\ast b_j,\;r_{12}\ast r_{23}:=\sum_{i,j}a_i\otimes b_i\ast a_j\otimes b_j,\;
r_{12}\ast r_{31}:=\sum_{i,j}a_i\ast b_j\otimes b_i\otimes a_j.
\end{align*}\end{small}

\section{Representations and matched pairs of Leibniz-dendriform algebras}

In this section, we first recall some basic results on Leibniz algebras and
Leibniz-dendriform algebras. We then study the representations and matched pairs of
 Leibniz-dendriform algebras. Furthermore, we show that a matched pair of Leibniz-dendriform
  algebras is equivalent to the phase space of a Leibniz algebra,
  and also corresponds to a Manin triple of Leibniz-dendriform algebras.

 A Leibniz algebra is a vector space $A$ equipped with a binary operation $\circ$ satisfying
\begin{equation*}x\circ(y\circ z)=(x\circ y)\circ z+y\circ(x\circ z)\end{equation*} for all $x,y,z\in A$.

A {\bf representation (bimodule)} of a Leibniz algebra $(A,\circ)$ is a triple $(V,l_{\circ},r_{\circ})$,
where $V$ is a vector space and $l_{\circ},r_{\circ}: A\rightarrow \text{End}(V)$ are linear maps satisfying
\begin{align}
\label{Lr1} &l_{\circ}(x\circ y)=l_{\circ}(x)l_{\circ}(y)v-l_{\circ}(y)l_{\circ}(x),\\
\label{Lr2}&l_{\circ}(x)r_{\circ}(y)-r_{\circ}(y)l_{\circ}(x)=r_{\circ}(x\circ y),\\
\label{Lr3} &r_{\circ}(y)r_{\circ}(x)=-r_{\circ}(y)l_{\circ}(x)
\end{align} for all $ x,y\in A.$

Let $(A,\circ)$ be a Leibniz algebra. If $(V,l_{\circ},r_{\circ})$ is a representation of $(A,\circ)$, then
$(V^{*},l_{\circ}^{*},-l_{\circ}^{*}-r_{\circ}^{*})$ is also a representation of $(A,\circ)$.

A symplectic Leibniz algebra is a Leibniz algebra $(A,\circ)$ together with a non-degenerate symmetric bilinear form
$\omega$ satisfying the following condition for all $x,y,z\in A$
\begin{equation}
\label{Bs}\omega(z,x\circ y)=-\omega(y,x\circ z)+\omega(x,y\circ z+z\circ y).
\end{equation}

\begin{defi}\cite{34}
Let $(A,\circ_A)$ be a Leibniz algebra and $A^{*}$ its dual space. If there is a
Leibniz algebra structure $\circ$ on the direct sum vector space $A\oplus A^{*}$ such that
$(A\oplus A^{*},\circ,\omega)$ is a symplectic Leibniz algebra, where $\omega$ is given by the following equation:
 \begin{equation}\label{C2}\omega(x+a, y+b) =\langle x,b\rangle+\langle a,y\rangle, \forall~x, y\in A_1,~a, b\in A_{1}^{*},\end{equation}
and both $(A,\circ_A)$ and $(A^{*},\circ|A^{*})$ are Leibniz subalgebras of $(A\oplus A^{*},\circ)$, then the
the symplectic Leibniz
algebra $(A\oplus A^{*},\circ,\omega)$ is called a phase space of the Leibniz algebra $(A,\circ_A)$.
\end{defi}

\begin{defi} \cite{28}
 A {\bf Leibniz-dendriform algebra} is a vector space $A$ together with two binary operations
$\succ,\prec: A\otimes A \rightarrow A$ satisfying
\begin{align} \label{Ld1}
&(x\circ y)\succ z=x\succ(y\succ z)-y\succ(x\succ z),\\
\label{Ld2}&y\prec(x\circ z)+(x\succ y)\prec z=x\succ (y\prec z),\\
\label{Ld3}&x\prec(y\circ z)=(x\prec y)\prec z+y\succ(x\prec z),
\end{align}
for all $x,y,z\in A$, where $x\circ y=x\succ y+x\prec y$.
  $(A,\circ)$ is a Leibniz algebra, which is called the associated Leibniz algebra of $(A,\succ,\prec)$
  and $(A,\succ,\prec)$ is called a compatible Leibniz-dendriform algebra on $(A,\circ)$
\end{defi}

Define $x\circ y=x\succ y+x\prec y$ and $x\odot y=x\succ y+y\prec x$.
By Eqs.~(\ref{Ld1})-(\ref{Ld3}), we have
\begin{align} &\label{Ld4}
(x\succ y+y\prec x)\prec z=0,\ \ \  (x\circ y+y\circ x)\succ z=0, \\&
\label{Ld5}
y\odot (x\odot z)-(y\circ x)\odot z=x\odot(y\succ z)=y\succ (x\odot z)-(y\circ x)\odot z.
\end{align}

\begin{ex}
Let $(A,\succ,\prec)$ be a 1-dimensional Leibniz-dendriform algebra with a basis $\{ e\}$.
Suppose that $e\succ e=pe,~e\prec e=qe$ with $p,q\in k$. Then
$(p+q)q=0,~(p+q)p=0$. Thus, $p=-q$.
\end{ex}

\begin{ex} \label{E1}
Let $A$ be a 2-dimensional vector space with a basis $\{e_1, e_2\}$. Define two bilinear
maps $\succ,\prec:A \otimes A \longrightarrow A$ respectively by
	\begin{align*}
		&e_1\succ e_1=e_1,\ \ \ e_1\prec e_1=-e_1, \ \ \ e_1\succ e_2=e_2, \\&
e_2\succ e_1=e_2,\ \ \ e_1\prec e_2=-e_2, \ \ \ e_2\prec e_1=-e_2,\ \ \ e_2\succ e_2=e_2\prec e_2=0.
	\end{align*}
 By direct calculation, $(A,\succ,\prec)$ is a Leibniz-dendriform algebra.
\end{ex}

\begin{ex} \label{E2}
Let $A$ be a 2-dimensional vector space over the real field $\mathbb{ R}$ with a basis $\{e_1, e_2\}$. Define two bilinear
maps $\succ,\prec:A \otimes A \longrightarrow A$ respectively by
	\begin{align*}
		&e_1\succ e_1=e_1,\ \ \ e_1\prec e_1=-e_1, \ \ \ e_1\succ e_2=e_2, \\&
e_2\succ e_1=0,\ \ \ e_1\prec e_2=0, \ \ \ e_2\prec e_1=-e_2,\ \ \ e_2\succ e_2=e_2\prec e_2=0.
	\end{align*}
 By a direct computation, $(A,\succ,\prec)$ is a Leibniz-dendriform algebra.
\end{ex}

\begin{defi}\cite{34}
A quadratic Leibniz-dendriform algebra is a
Leibniz-dendriform algebra $(A,\succ,\prec)$ equipped with a non-degenerate symmetric bilinear form $\omega$ such that the following
invariant conditions hold for all $x, y, z \in A$,
\begin{equation} \label{C1}\omega (x \prec y, z)=\omega(x, y\circ z+z\circ y), \ \  \
\omega(x \succ y, z)=-\omega( y,x\circ z).\end{equation}
\end{defi}

\begin{thm} \cite{34}\label{Sq} Let $(A,\circ,\omega)$ be a symplectic Leibniz algebra.
 Then there exists a compatible Leibniz-dendriform algebra structure $(A,\succ,\prec)$
 on $(A,\circ)$ defined by Eq.~(\ref{C1}),
such that $(A,\circ)$ is the associated Leibniz algebra of $(A,\succ,\prec)$. This Leibniz-dendriform algebra is called
the compatible Leibniz-dendriform algebra of $(A,\circ,\omega)$. Moreover, $(A,\succ,\prec,\omega)$ is a quadratic Leibniz-dendriform algebra.
Conversely, assume that $(A,\succ,\prec,\omega)$ is a quadratic Leibniz-dendriform algebra.
Then $(A,\circ,\omega)$ is a symplectic Leibniz algebra.
\end{thm}

\begin{defi}
Let $(A,\succ,\prec)$ be a Leibniz-dendriform algebra, $V$ a vector space and
$l_{\succ},r_{\succ},l_{\prec},r_{\prec}:A\rightarrow \text{End}(V)$ be linear maps.
 $(V,l_{\succ},r_{\succ},l_{\prec},r_{\prec})$ is called a \textbf{representation (bimodule)} of $(A,\succ,\prec)$
  if the following conditions hold:
\begin{flalign}
& \label{R1}l_{\succ}(x\circ y)=l_{\succ}(x)l_{\succ}(y)-l_{\succ}(y)l_{\succ}(x), \\
&\label{R2}r_{\succ}(y\succ z)-l_{\succ}(y)r_{\succ}(z)=r_{\succ}(z)r_{\circ}(y), \\
&\label{R3}l_{\succ}(x)r_{\succ}(z)-r_{\succ}(x\succ z)=r_{\succ}(z)l_{\circ}(x), \\
&\label{R4}l_{\succ}(x)l_{\prec}(y)=l_{\prec}(x\succ y)+l_{\prec}(y)l_{\circ}(x), \\
&\label{R5}r_{\succ}(y\prec z)=r_{\prec}(z)r_{\succ}(y)+l_{\prec}(y)r_{\circ}(z), \\
&\label{R6}l_{\succ}(x)r_{\prec}(z)=r_{\prec}(z)l_{\succ}(x)+r_{\prec}(x\circ z), \\
&\label{R7}l_{\prec}(x)l_{\circ}(y)=l_{\prec}(x\prec y)+l_{\succ}(y)l_{\prec}(x), \\
&\label{R8}r_{\prec}(y\circ z)=r_{\prec}(z)r_{\prec}(y)+l_{\succ}(y) r_{\prec}(z),\\
&\label{R9}l_{\prec}(x)r_{\circ}(z)=r_{\prec}(z)l_{\prec}(x)+r_{\succ}(x\prec z),
\end{flalign}
for all $x,y\in A$, where $l_{\circ}=l_{\succ}+l_{\prec},  \ r_{\circ}=r_{\succ}+r_{\prec}$ and
$x\circ y=x\succ y+x\prec y$.
\end{defi}
By Eqs.~(\ref{R3})-(\ref{R9}), we get
\begin{align}&\label{R10}
r_{\succ}(x\odot y)-l_{\odot}(x)r_{\succ}(y)-r_{\odot}(y)r_{\circ}(x)=0,\\&
\label{R11} l_{\prec}(x\odot y)=0,\ \ \ r_{\prec}(x)l_{\odot}(y)=0,\ \ \ r_{\prec}(x)l_{\star}(x)=0,
\end{align}
where $l_{\star}=l_{\circ}+r_{\circ}$ and $l_{\odot}=l_{\succ}+r_{\prec}$.

Let $A$ and $V$ be vector spaces. For
a linear map $f: A \longrightarrow \hbox{End} (V)$, define a linear
map $f^{*}: A \longrightarrow \hbox{End} (V^{*})$ by $\langle
f^{*}(x)u^{*},v\rangle=-\langle u^{*},f(x)v\rangle$ for all $x\in A,
u^{*}\in V^{*}, v\in V$, where $\langle \ , \ \rangle$ is the usual
pairing between $V$ and $V^{*}$.

\begin{pro} \label{Dr} Let $(A,\succ,\prec)$ be a Leibniz-dendriform algebra and $(V,l_{\succ}, r_{\succ},l_{\prec},r_{\prec})$ be its
representation. Then
\begin{enumerate}
	\item $(V,l_{\succ},r_{\prec})$ is a representation of the associated
Leibniz algebra $(A,\circ)$.
 \item $(V,l_{\circ},r_{\circ})$ is a representation of the
 associated Leibniz algebra $(A,\circ)$.
	\item $(V^*,l_{\circ}^*,r_{\odot}^*,
-l_{\prec}^*,-l_{\star}^*)$
is a representation of $(A,\succ,\prec)$. We call it the {\bf dual representation}.
	\item $(V^{*},l_{\circ}^*,-r_{\circ}^*-l_{\circ}^*)
$ is a representation of the
 associated Leibniz algebra $(A,\circ)$.
\item $(V^{*},l_{\succ}^{*},-l_{\odot}^{*})$ is a
	representation of the
 associated Leibniz algebra $(A,\circ)$,
\end{enumerate}
where $l_{\circ}=l_{\prec}+l_{\succ},~r_{\circ}=r_{\prec}+r_{\succ},
 l_{\star}=l_{\circ}+r_{\circ},~l_{\odot}=l_{\succ}+r_{\prec}$ and $r_{\odot}=r_{\succ}+l_{\prec}.$
\end{pro}

\begin{proof} This can be verified through direct computations.
\end{proof}

\begin{ex} Let $(A,\succ,\prec)$ be a Leibniz-dendriform algebra. Then
\begin{enumerate}
\item $ (A,L_{\succ},R_{\succ},L_{\prec},R_{\prec})$ is a representation of $(A,\succ,\prec)$,
which is called the {\bf regular representation} of $(A,\succ,\prec)$.
Moreover, $(A^{*},L_{\prec}^{*}+L_{\succ}^{*},R_{\succ}^{*}+L_{\prec}^{*},
-L_{\prec}^*,-(L_{\prec}^{*}+L_{\succ}^{*}+R_{\prec}^{*}+R_{\succ}^{*}))$
is the dual representation of $(A,L_{\succ},R_{\succ},L_{\prec},R_{\prec})$.
\item $ (A,L_{\succ},R_{\prec})$ and $(A^{*},L_{\succ}^{*},-(L_{\succ}^{*}+R_{\prec}^{*}))$ are both representations of the associated
Leibniz algebra $(A,\circ)$.
\end{enumerate}
\end{ex}

Before turning to matched pairs of Leibniz-dendriform algebras,
we first revisit the matched pairs of Leibniz algebras that were introduced in \cite{28}.

\begin{pro} \cite{28} \label{d1}
Let $(A,\circ)$ and $(B,\bullet)$ be Leibniz algebras. If $(B,l_A,r_A)$ is a representation
of $(A,\circ)$, $(A,l_B,r_B)$ is a representation of $(B,\bullet)$ and the following conditions are satisfied:
\begin{flalign}\label{Lm1}
&r_A(x)(a\bullet b)-a\bullet r_A(x)(b)+b\bullet r_A(x)(a)-r_A(l_B(b)x)a+r_A(l_B(a)x)b=0,\\&
\label{Lm2}
l_A(a)(x\circ y)-(l_A(x)a)\bullet b-a\bullet (l_A(x)b)-l_A(r_B(a)x)b-r_A(r_B(b)x)a=0,
\\&
\label{Lm3}(l_A(x)a)\bullet b+l_A(r_B(a)x)b+(r_{A}(x)a)\bullet b+l_A(l_B(a)x)b=0,\\&
\label{Lm4}r_B(a)(x\circ y)-x\circ (r_B(a)y)+y\circ (r_B(a)x)-r_B(l_A(y)a)x+r_B(l_A(x)a)y=0
,\\&
\label{Lm5}l_B(a)(x\circ y)-(l_B(a)x)\circ y-x\circ (l_B(a)y)-l_B(r_A(x)a)y-r_B(r_A(y)a)x=0
,\\&
\label{Lm6}(l_B(a)x)\circ y+l_B(r_A(x)a)y+(r_B(a)x)\circ y+l_B(l_A(x)a)y=0,  
\end{flalign}
for all $x,y\in A,~a,b\in B.$ Then there is a Leibniz algebra structure on the direct sum $A\oplus B$ of
 the underlying vector spaces of $A$ and $B$ given by
\begin{align*}
(x+a)\cdot (y+b)=(x\circ y+l_B(a)y+r_B(b)x)+(a\bullet b+l_A(x)b+r_A(y)a),\;a, b\in A,\;x, y\in B.
\end{align*}
 $(A,B,l_A,r_A,l_B,r_B)$ satisfying the above conditions is called a \textbf{matched pair of Leibniz algebras.}
 Conversely, any Leibniz algebra that can be decomposed into a direct sum of
 two Leibniz subalgebras is obtained from a matched pair of Leibniz algebras.
\end{pro}

\begin{rmk} If the linear maps $l_B,r_B: B\rightarrow \text{End}(A)$ are trivial, then the matched pair $(A,B,l_A,r_A,l_B,r_B)$
reduces to an $A$-Leibniz algebra. That is,
$(B,l_A,r_A)$
is an $A$-Leibniz algebra \cite{07}.
\end{rmk}

\begin{pro} \label{M0} Let $(A_{1},\succ_{1},\prec_{1})$ and
$(A_{2},\succ_{2},\prec_{2})$ be two Leibniz-dendriform algebras. Assume that there are linear maps
$l_{\succ_1},r_{\succ_1},l_{\prec_1},r_{\prec_1}:A_1\longrightarrow \hbox{End}(A_2)$
and $l_{\succ_2},r_{\succ_2},l_{\prec_2},r_{\prec_2}:A_2\longrightarrow \hbox{End}(A_1)$ such that
$(A_2,l_{\succ_1},r_{\succ_1},l_{\prec_1},r_{\prec_1})$ is a representation of $(A_1,\succ_1,\prec_1)$
and $(A_1,l_{\succ_2},r_{\succ_2},l_{\prec_2},r_{\prec_2})$ is a representation of $(A_2,\succ_2,\prec_2)$.
Moreover, the following compatible conditions hold for all $x,y\in A_{1}$ and $a,b\in A_{2}$:
\begin{align}
&\label{Lm1}r_{\succ_2}(a)(x\circ_1 y)=x\succ_1 r_{\succ_2}(a)y-y\succ_1 r_{\succ_2}(a)x+r_{\succ_2}(l_{\succ_1}(y)a)x-r_{\succ_2}(l_{\succ_1}(x)a)y,\\
&\label{Lm2}(r_{\circ_2}(a)x)\succ_1 y+l_{\succ_2}(l_{\circ_1}(x)a)y=x\succ_{1}l_{\succ_2}(a)y+r_{\succ_2}(r_{\succ_1}(y)a)x-l_{\succ_2}(a)(x\succ_1 y),\\
&\label{Lm3}l_{\succ_2}(a)(x\succ_1 y)-x\succ_1(l_{\succ_2}(a)y)-r_{\succ_2}(r_{\succ_1}(y)a)x=(l_{\circ_2}(a)x)\succ_1 y+l_{\succ_2}(r_{\circ_1}(x)a)y,\\
&\label{Lm4}r_{\succ_1}(x)(a\circ_2 b)=a\succ_2 r_{\succ_1}(x)b-b\succ_2 r_{\succ_1}(x)a+r_{\succ_1}(l_{\succ_2}(b)x)a-r_{\succ_1}(l_{\succ_2}(a)x)b,\\
&\label{Lm5}(r_{\circ_1}(x)a)\succ_2 b+l_{\succ_1}(l_{\circ_2}(a)x)b=a\succ_{2}l_{\succ_1}(x)b+r_{\succ_1}(r_{\succ_2}(b)x)a-l_{\succ_1}(x)(a\succ_2 b),\\
&\label{Lm6}l_{\succ_1}(x)(a\succ_2 b)-a\succ_2(l_{\succ_1}(x)b)-r_{\succ_1}(r_{\succ_2}(b)x)a=(l_{\circ_1}(x)a)\succ_2 b+l_{\succ_1}(r_{\circ_2}(a)x)b,\\
&\label{Lm7}x\succ_1(r_{\prec_2}(a)y)+r_{\succ_2}(l_{\prec_1}(y)a)x-r_{\prec_2}(a)(x\succ_1 y)=y\prec_1(r_{\circ_2}(a)x)+r_{\prec_2}(l_{\circ_1}(x)a)y,\\
&\label{Lm8}x\succ_1(l_{\prec_2}(a)y)+r_{\succ_2}(r_{\prec_1}(y)a)x-(r_{\succ_2}(a)x)\prec_1 y-l_{\prec_2}(l_{\succ_1}(x)a)y
=l_{\prec_2}(a)(x\circ_1y),\\
&\label{Lm9}l_{\succ_2}(a)(x\prec_1 y)-(l_{\succ_2}(a)x)\prec_1 y-l_{\prec_2}(r_{\succ_1}(x)a)y=x\prec_1(l_{\circ_2}(a)y)+r_{\prec_2}(r_{\circ_1}(y)a)x,\\
&\label{Lm10}a\succ_2(r_{\prec_1}(x)b)+r_{\succ_1}(l_{\prec_2}(b)x)a-r_{\prec_1}(x)(a\succ_2 b)=b\prec_2(r_{\circ_1}(x)a)+r_{\prec_1}(l_{\circ_2}(a)x)b,\\
&\label{Lm11}a\succ_2(l_{\prec_1}(x)b)+r_{\succ_1}(r_{\prec_2}(b)x)a-(r_{\succ_1}(x)a)\prec_2 b-l_{\prec_1}(l_{\succ_2}(a)x)b
=l_{\prec_1}(x)(a\circ_2 b),\\
&\label{Lm12}l_{\succ_1}(x)(a\prec_2 b)-(l_{\succ_1}(x)a)\prec_2 b-l_{\prec_1}(r_{\succ_2}(a)x)b=a\prec_2(l_{\circ_1}(x)b)+r_{\prec_1}(r_{\circ_2}(b)x)a,\\
&\label{Lm13}x\prec_1(r_{\circ_2}(a)y)+r_{\prec_2}(l_{\circ_1}(y)a)x=r_{\prec_2}(a)(x\prec_1y)+y\succ_1(r_{\prec_2}(a)x)+r_{\succ_2}(l_{\prec_1}(x)a)y,\\
&\label{Lm14}x\prec_1(l_{\circ_2}(a)y)+r_{\prec_2}(r_{\circ_1}(y)a)x=(r_{\prec_2}(a)x)\prec_1y+l_{\prec_2}(l_{\prec_1}(x)a)y+l_{\succ_2}(a)(x\prec_1y),\\
&\label{Lm15}(l_{\prec_2}(a)x)\prec_1 y+l_{\prec_2}(r_{\prec_1}(x)a)y+x\succ_1(l_{\prec_2}(a)y)+r_{\succ_2}(r_{\prec_1}(y)a)x=l_{\prec_2}(a)(x\circ_1y),\\
&\label{Lm16}a\prec_2(r_{\circ_1}(x)b)+r_{\prec_1}(l_{\circ_2}(b)x)a=r_{\prec_1}(x)(a\prec_2 b)+b\succ_2(r_{\prec_1}(x)a)+r_{\succ_1}(l_{\prec_2}(a)x)b,\\
&\label{Lm17}a\prec_2(l_{\circ_1}(x)b)+r_{\prec_1}(r_{\circ_2}(b)x)a=(r_{\prec_1}(x)a)\prec_2 b+l_{\prec_1}(l_{\prec_2}(a)x)b+l_{\succ_1}(x)(a\prec_2 b),\\
&\label{Lm18}(l_{\prec_1}(x)a)\prec_2 b+l_{\prec_1}(r_{\prec_2}(a)x)b+a\succ_2(l_{\prec_1}(x)b)+r_{\succ_1}(r_{\prec_2}(b)x)a=l_{\prec_1}(x)(a\circ_2 b),
\end{align}
where
$l_{\circ_1}=l_{\succ_1}+l_{\prec_1},\ l_{\circ_2}=l_{\succ_2}+l_{\prec_2},\
r_{\circ_1}=r_{\succ_1}+r_{\prec_1},\ r_{\circ_2}=r_{\succ_2}+r_{\prec_2}.$
Define two binary operations $\succ$ and $\prec$ on the direct sum $A_1\oplus A_2$ of the underlying vector spaces
of $A_1$ and $A_2$ by
\begin{align*}&(x+a)\succ(y+b)=x\succ_{1}y+l_{\succ_2}(a)y+r_{\succ_2}(b)x+a\succ_{2}b+l_{\succ_1}(x)b+r_{\succ_1}(y)a,\\&
(x+a)\prec(y+b)=x\prec_{1}y+l_{\prec_2}(a)y+r_{\prec_2}(b)x+a\prec_{2}b+l_{\prec_1}(x)b+r_{\prec_1}(y)a.
\end{align*}
Then $(A_1\oplus A_2,\succ,\prec)$ is a Leibniz-dendriform algebra.
Denote this Leibniz-dendriform algebra by $A_1\bowtie A_2$ and
$(A_{1},A_{2},l_{\succ_1},r_{\succ_1},l_{\prec_1},r_{\prec_1},l_{\succ_2},r_{\succ_2},l_{\prec_2},r_{\prec_2})$
satisfying the above conditions is called a {\bf matched pair of
Leibniz-dendriform algebras}. Conversely, any Leibniz-dendriform algebra that can be decomposed into a
direct sum of two Leibniz-dendriform subalgebras is obtained from a matched pair of Leibniz-dendriform algebras.
\end{pro}
\begin{proof} To show that $(A_1 \oplus A_2, \ast, \circ)$ is a Leibniz-dendriform algebra,
 it suffices to verify that Eqs.~(\ref{Ld1})–(\ref{Ld3}) are satisfied for all
 $x + a, y + b, z + c \in A_1 \oplus A_2$, where $x, y, z \in A_1$ and $a, b, c \in A_2$.
 This can be established through a direct computation, as outlined in the following cases:
\begin{enumerate}
\item Eq.~(\ref{Ld1}) holds on $(x, y, a)$ if and only if Eqs.~(\ref{R1}) and (\ref{Lm1}) hold.
\item Eq.~(\ref{Ld1}) holds on $(x, a, y)$ if and only if Eqs.~(\ref{R3}) and (\ref{Lm2}) hold.
\item Eq.~(\ref{Ld1}) holds on $(a, x, y)$ if and only if Eqs.~(\ref{R2}) and (\ref{Lm3}) hold.
\item Eq.~(\ref{Ld1}) holds on $(a,b, x)$ if and only if Eqs.~(\ref{R1}) and (\ref{Lm4}) hold.
\item Eq.~(\ref{Ld1}) holds on $(a, x, b)$ if and only if Eqs.~(\ref{R3}) and (\ref{Lm5}) hold.
\item Eq.~(\ref{Ld1}) holds on $(x, a, b)$ if and only if Eqs.~(\ref{R2}) and (\ref{Lm6}) hold.
\item Eq.~(\ref{Ld2}) holds on $(x, y,a)$ if and only if Eqs.~(\ref{R4}) and (\ref{Lm7}) hold.
\item Eq.~(\ref{Ld2}) holds on $(x,a, y)$ if and only if Eqs.~(\ref{R6}) and (\ref{Lm8}) hold.
\item Eq.~(\ref{Ld2}) holds on $(a, x, y)$ if and only if Eqs.~(\ref{R5}) and (\ref{Lm9}) hold.
\item Eq.~(\ref{Ld2}) holds on $( a, b,x)$ if and only if Eqs.~(\ref{R4}) and (\ref{Lm10}) hold.
\item Eq.~(\ref{Ld2}) holds on $(a, x, b)$ if and only if Eqs.~(\ref{R6}) and (\ref{Lm11}) hold.
\item Eq.~(\ref{Ld2}) holds on $(a, b, x)$ if and only if Eqs.~(\ref{R5}) and (\ref{Lm12}) hold.
\item Eq.~(\ref{Ld3}) holds on $(x, y,a)$ if and only if Eqs.~(\ref{R7}) and (\ref{Lm13}) hold.
\item Eq.~(\ref{Ld3}) holds on $(x,a, y)$ if and only if Eqs.~(\ref{R9}) and (\ref{Lm14}) hold.
\item Eq.~(\ref{Ld3}) holds on $(a, x, y)$ if and only if Eqs.~(\ref{R8}) and (\ref{Lm15}) hold.
\item Eq.~(\ref{Ld3}) holds on $( a, b,x)$ if and only if Eqs.~(\ref{R7}) and (\ref{Lm16}) hold.
\item Eq.~(\ref{Ld3}) holds on $(a, x, b)$ if and only if Eqs.~(\ref{R9}) and (\ref{Lm17}) hold.
\item Eq.~(\ref{Ld3}) holds on $(a, b, x)$ if and only if Eqs.~(\ref{R8}) and (\ref{Lm18}) hold.
\end{enumerate}
\end{proof}

\begin{rmk} \label{La} \begin{enumerate}
\item If the linear maps
 $l_{\succ_2},r_{\succ_2},l_{\prec_2},r_{\prec_2}: A_{2}\rightarrow \text{End}(A_1)$ are trivial, then the matched pair $(A_{1},A_{2},l_{\succ_1},r_{\succ_1},l_{\prec_1},r_{\prec_1},l_{\succ_2},r_{\succ_2},l_{\prec_2},r_{\prec_2})$
reduces to an A-Leibniz-dendriform algebra. That is,
$(A_{2},l_{\succ_1},r_{\succ_1},l_{\prec_1},r_{\prec_1})$
is an A-Leibniz-dendriform algebra.
For the completeness, we list it in detail. Let $(A_{1},\succ_{1},\prec_{1})$ and
$(A_{2},\succ_{2},\prec_{2})$ be two Leibniz-dendriform algebras. Assume that there are linear maps
$l_{\succ_1},r_{\succ_1},l_{\prec_1},r_{\prec_1}:A_1\longrightarrow \hbox{End}(A_2)$
 such that
$(A_2,l_{\succ_1},r_{\succ_1},l_{\prec_1},r_{\prec_1})$ is a representation of $(A_1,\succ_1,\prec_1)$.
Moreover, the following compatible conditions hold:
\begin{align*}&(r_{\circ_1}(x)a)\succ_2 b=a\succ_{2}l_{\succ_1}(x)b-l_{\succ_1}(x)(a\succ_2 b)=-(l_{\circ_1}(x)a)\succ_2 b,\\
&r_{\succ_1}(x)(a\circ_2 b)=a\succ_2 r_{\succ_1}(x)b-b\succ_2 r_{\succ_1}(x)a,\\
&a\succ_2(r_{\prec_1}(x)b)-r_{\prec_1}(x)(a\succ_2 b)=b\prec_2(r_{\circ_1}(x)a),\\
&a\succ_2(l_{\prec_1}(x)b)-(r_{\succ_1}(x)a)\prec_2 b
=l_{\prec_1}(x)(a\circ_2 b),\\
&l_{\succ_1}(x)(a\prec_2 b)-(l_{\succ_1}(x)a)\prec_2 b=a\prec_2(l_{\circ_1}(x)b),\\
&a\prec_2(r_{\circ_1}(x)b)=r_{\prec_1}(x)(a\prec_2 b)+b\succ_2(r_{\prec_1}(x)a),\\
&a\prec_2(l_{\circ_1}(x)b)=(r_{\prec_1}(x)a)\prec_2 b+l_{\succ_1}(x)(a\prec_2 b),\\
&(l_{\prec_1}(x)a)\prec_2 b+a\succ_2(l_{\prec_1}(x)b)=l_{\prec_1}(x)(a\circ_2 b)
\end{align*}
for all $x,y\in A_1,a,b\in A_2.$ Then $(A_{2},l_{\succ_1},r_{\succ_1},l_{\prec_1},r_{\prec_1})$ is called
  an A-Leibniz-dendriform algebra.
  \item If the linear maps
 $l_{\succ_2},r_{\succ_2},l_{\prec_2},r_{\prec_2}: A_{2}\rightarrow \text{End}(A_1)$ and $(\succ_{2},\prec_{2})$ are trivial, then the matched pair $(A_{1},A_{2},l_{\succ_1},r_{\succ_1},l_{\prec_1},r_{\prec_1},l_{\succ_2},r_{\succ_2},l_{\prec_2},r_{\prec_2})$
reduces to the semidirect product. Denote it by $A_{1}\ltimes A_2$.
  \end{enumerate}
\end{rmk}

\begin{cor}\label{Ma} If
$(A_{1},A_{2},l_{\succ_1},r_{\succ_1},l_{\prec_1},r_{\prec_1},l_{\succ_2},r_{\succ_2},l_{\prec_2},r_{\prec_2})$ is
a matched pair of Leibniz-dendriform algebras, then
$(A_{1},A_{2},l_{\succ_1}+l_{\prec_1},r_{\succ_1}+r_{\prec_1},l_{\succ_2}+l_{\prec_2},r_{\succ_2}+r_{\prec_2})$
is a matched pair of Leibniz algebras.
\end{cor}

\begin{proof}
In view of Proposition \ref{M0}, there is a Leibniz-dendriform algebra $(A_1\bowtie A_2,\succ,\prec)$, whose associated
Leibniz algebra is defined by
\begin{align*}&(x+a)\circ(y+b)=(x+a)\succ(y+b)+(x+a)\prec(y+b)
\\=&x\succ_{1}y+l_{\succ_2}(a)y+r_{\succ_2}(b)x+a\succ_{2}b+l_{\succ_1}(x)b+r_{\succ_1}(y)a
+x\prec_{1}y+l_{\prec_2}(a)y+r_{\prec_2}(b)x\\&+a\prec_{2}b+l_{\prec_1}(x)b+r_{\prec_1}(y)a
\\=&x\circ_{1}y+(l_{\succ_2}+l_{\prec_2})(a)y+(r_{\succ_2}+r_{\prec_2})(b)x+a\circ_{2}b+(l_{\succ_1}+l_{\prec_1})(x)b+(r_{\succ_1}+r_{\prec_1})(y)a
.\end{align*}
Combining Proposition \ref{d1}, we get the conclusion.
\end{proof}

\begin{defi}\cite{34}
A Manin triple of Leibniz-dendriform algebras is a triple
of Leibniz-dendriform algebras $(A,A_{1},A_{2})$ satisfying
\begin{enumerate}
\item $A=A_{1}\oplus A_{2}$ as vector spaces.
\item $(A,\omega)$ is a quadratic Leibniz-dendriform algebra.
\item $A_{1},A_{2}$ are isotropic subalgebras of $A$, that is,
$\omega(x,y) =\omega(a,b) =0$ for all $x,y\in A_1$ and $a,b\in A_2$.
\end{enumerate}
\end{defi}

\begin{thm} \label{Mp1}
 Let $(A,\succ_{A},\prec_{A})$ and $(A^{*},\succ_{A^{*}},\prec_{A^{*}})$ be two Leibniz-dendriform algebras and their associated
 Leibniz algebras be $(A,\circ)$ and $(A^{*},\ast)$ respectively. Then the following conditions are equivalent:
\begin{enumerate}
\item $(A\oplus A^{*},\circ,\omega)$ is a phase space of the Leibniz algebra $(A,\circ)$.
\item $(A\oplus A^{*},A, A^{*},\omega)$ is a Manin triple of Leibniz-dendriform algebras with the
bilinear form given by Eq.~\eqref{C2}
and the isotropic subalgebras are $A$ and $ A^{*}$.
 \item $(A,A^{*}, L_{\circ}^{*},R_{\succ_A}^{*}+L_{\prec_A}^{*},-L_{\prec_{A}}^*,-L_{\circ}^{*}R_{\circ}^{*},
 L_{\ast}^{*},R_{\succ_A^{*}}^{*}+L_{\prec_A^{*}}^{*},-L_{\prec_{A^{*}}}^*,-L_{\ast}^{*}-R_{\ast}^{*})$
 is a matched pair of Leibniz-dendriform algebras.
 \item $(A, A^*,L_{\succ_{A}}^*,-L_{\succ_A}^{*}-R_{\prec_A}^{*},L_{\succ_{A^*}}^{*},-L_{\succ_A^{*}}^{*}-R_{\prec_A^{*}}^{*})$
 is a matched pair of Leibniz algebras.
 \end{enumerate}
\end{thm}
\begin{proof}
 $(a)\Longleftrightarrow (b)$. It follows from Theorem 3.10 \cite{34}.

$(b)\Longrightarrow (c)$ By Proposition \ref{M0}, there are linear maps
$l_{\succ_A},r_{\succ_A},l_{\prec_A},r_{\prec_A}:A\longrightarrow \hbox{End}(A^{*})$ and
$l_{\succ_{A^{*}}},r_{\succ_{A^{*}}},l_{\prec_{A^{*}}},r_{\prec_{A^{*}}}:A^{*}\longrightarrow \hbox{End}(A)$ such that
$(A,A^{*},l_{\succ_A},r_{\succ_A},l_{\prec_A},r_{\prec_A}, l_{\succ_{A^{*}}},r_{\succ_{A^{*}}},l_{\prec_{A^{*}}},r_{\prec_{A^{*}}})$
is a matched pair of Leibniz-dendriform algebras and
\begin{align*}
&x\succ b=r_{\succ_{A^{*}}}(b)x+l_{\succ_A}(x)b,\ \ \ b\succ x=l_{\succ_{A^{*}}}(b)x+r_{\succ_A}(x)b,
\\&
x\prec b=r_{\prec_{A^{*}}}(b)x+l_{\prec_A}(x)b,\ \ \ b\prec x=l_{\prec_{A^{*}}}(b)x+r_{\prec_A}(x)b,
\end{align*}
for all $x\in A$ and $b\in A^{*}$.
Then we obtain,
\begin{align*}\langle l_{\succ_A}(x)b,y\rangle &=\omega (x\succ b,y)=-\omega(b,x\circ y)
\\&=-\langle b,(L_{\succ_A}(x)+L_{\prec_A}(x))y
\rangle\\&=\langle (L_{\succ_A}^{*}+L_{\prec_A}^{*})(x)b,y
 \rangle,\end{align*}
which indicates that
$l_{\succ_A}=L_{\succ_A}^{*}+L_{\prec_A}^{*}=L_{\circ}^{*}$.
Analogously, $r_{\succ_A}=R_{\succ_A}^{*}+L_{\prec_A}^{*},~l_{\prec_A}=-L_{\prec_A}^{*},
~r_{\prec_A}=L_{\circ}^{*}+R_{\circ}^{*},~
l_{\succ_{A^{*}}}=L_{\ast}^{*}
,~r_{\succ_{A^{*}}}=R_{\succ_A^{*}}^{*}+L_{\prec_A^{*}}^{*}
,~l_{\prec_{A^{*}}}=-L_{\prec_A^{*}}^{*}
,~r_{\prec_{A^{*}}}=L_{\ast}^{*}+R_{\ast}^{*}$. Thus, Item (b) holds.

$(c)\Longrightarrow (d)$ It can be obtained by Corollary \ref {Ma}.

$(d)\Longrightarrow (a)$  Assume that $(A, A^*,L_{\succ_{A}}^*,-L_{\succ_A}^{*}-R_{\prec_A}^{*},L_{\succ_{A^*}}^{*},-L_{\succ_A^{*}}^{*}-R_{\prec_A^{*}}^{*})$
is a matched pair of Leibniz algebras. Then $(A\bowtie A^{*},\circ)$
 is a Leibniz algebra and $A,A^*$ are Leibniz subalgebras of $A\bowtie A^{*}$, where
  $\circ$ is given by
\begin{equation*}x\circ a=-(L_{\succ_A^{*}}^{*}+R_{\prec_A^{*}}^{*})(a)x+L_{\succ_{A}}^*(x) a,\ \ \
a\circ x=L_{\succ_{A^*}}^{*}(a) x-(L_{\succ_A}^{*}+R_{\prec_A}^{*})(x)a, \ \ \forall ~x\in A, a\in A^{*}.\end{equation*}
Define a bilinear map on $A\oplus A^* $by Eq.~\eqref{C2}, we get for all $x,y\in A, a\in A^{*}$
\begin{align*}&\omega(a,x\circ y)+\omega(y,x\circ a)-\omega(x,y\circ a+a\circ y)
\\=&\omega(a,x\circ y)+\omega(y,-R_{\prec_{A}}^*(x) a)+\omega(x,R_{\prec_{A^*}}^{*}(a) y+R_{\prec_{A}}^*(y) a)
\\=&\langle a,x\circ y\rangle+\langle y,L_{\succ_{A}}^*(x) a\rangle
+\langle x,R_{\prec_{A}}^*(y) a\rangle
\\=&\langle a,x\circ y\rangle-\langle x\succ_A y, a\rangle
-\langle x\prec_A y, a\rangle
\\=&0,\end{align*}
that is, $\omega$ satisfies Eq.~\eqref{Bs}.
It follows that $(A\oplus A^{*},\circ,\omega)$ is a phase space of the Leibniz algebra $(A,\circ)$.
The proof is completed.
\end{proof}

\section{Leibniz-dendriform bialgebras  }
In this section, we introduce the notion of Leibniz-dendriform bialgebras.
 We demonstrate that such bialgebras are equivalent to specific matched pairs of Leibniz algebras and 
 to the phase spaces of Leibniz algebras.
The study of the coboundary case naturally leads to the Leibniz-dendriform Yang-Baxter equation (LD-YBE),
whose skew-symmetric solutions yield coboundary Leibniz-dendriform bialgebras.
Moreover, we introduce the notion of $\mathcal{O}$-operators on Leibniz-dendriform algebras,
which offers a  way to construct skew-symmetric solutions of the LD-YBE.

\subsection{Leibniz-dendriform bialgebras }
\begin{defi} \label{DB1} A Leibniz-dendriform coalgebra is a triple $(A,\Delta_{\succ},\Delta_{\prec})$, where
$A$ is a vector space and $\Delta_{\succ},\Delta_{\prec}:A\longrightarrow A\otimes A$ are linear maps such that
the following conditions hold:
\begin{align}&\label{Ca1}
( \Delta\otimes I)\Delta_{\succ}=(I\otimes\Delta_{\succ})\Delta_{\succ}-(\tau \otimes I)(I\otimes\Delta_{\succ})\Delta_{\succ},\\&
\label{Ca2}(I\otimes \Delta_{\prec})\Delta_{\succ}-(\Delta_{\succ}\otimes I)\Delta_{\prec}=
(\tau\otimes I)(I \otimes \Delta)\Delta_{\prec},\\&
\label{Ca3}( I\otimes \Delta)\Delta_{\prec}=( \Delta_{\prec}\otimes I)\Delta_{\prec}+(\tau\otimes I)(I \otimes \Delta_{\prec})\Delta_{\succ},\end{align}
where $\Delta=\Delta_{\succ}+\Delta_{\prec}$.
\end{defi}

\begin{defi} \label{DB2} A Leibniz-dendriform bialgebra is a quintuple $(A,\succ,\prec,\Delta_{\succ},\Delta_{\prec})$
such that $(A,$ \ \ \  $\succ,\prec)$ is a Leibniz-dendriform algebra, $(A,\Delta_{\succ},\Delta_{\prec})$ is
a Leibniz-dendriform coalgebra and the following compatible conditions hold:
\begin{small}
\begin{align}&\label{B1}
\Delta(x\odot y)-(I\otimes L_{\odot}(x)\Delta(y)+\tau(I\otimes L_{\odot}(x))\Delta(y)+
\tau (I\otimes R_{\odot}(y))\Delta_{\succ}(x)-(I\otimes R_{\odot}(y))\Delta_{\succ}(x)=0,\\&
\label{B2}\Delta(x\succ y)-(L_{\succ}(x)\otimes I+I\otimes L_{\succ}(x))\Delta (y)-(I\otimes R_{\succ}(y))\Delta_{\odot}(x)
+\tau(I\otimes R_{\odot}(y))\Delta_{\odot}(x)=0,\\&
\label{B3}(I\otimes R_{\succ}(y))\tau\Delta_{\prec}(x)-(R_{\prec}(x)\otimes I)\Delta(y)=0,\\&
\label{B4}(\Delta_{\succ} +\tau\Delta_{\prec})(x\circ y)-(I\otimes L_{\circ}(x)+L_{\succ}(x)\otimes I)\Delta_{\odot}(y)
+(I\otimes L_{\circ}(y)+L_{\succ}(y)\otimes I)\Delta_{\odot}(x)=0,\\&
\label{B5}\Delta_{\succ}(x\circ y)-(I\otimes R_{\circ}(y))\Delta_{\succ}(x)-(I\otimes L_{\circ}(x)+L_{\odot}(x)\otimes I)\Delta_{\succ}(y)
+(L_{\odot}(y)\otimes I)\Delta_{\odot}(x)=0,\\&
\label{B6}(I\otimes R_{\circ}(y))\tau\Delta_{\prec}(x)
-(R_{\prec}(x)\otimes I)\Delta_{\succ}(y)=0,
\end{align}
\end{small}
where $x\circ y=x \succ y+x\prec y,~x\odot y=x\succ y+y\prec x,~\Delta=\Delta_{\succ}+\Delta_{\prec}, ~\Delta_{\odot}=\Delta_{\succ}+\tau\Delta_{\prec}
$ and $R_{\circ}=R_{\prec}+R_{\succ},
~L_{\circ}=L_{\prec}+L_{\succ},~L_{\odot}=R_{\prec}+L_{\succ},~R_{\odot}=L_{\prec}+R_{\succ}$.
\end{defi}

\begin{rmk}
The triple $(A, \Delta_{\succ}, \Delta_{\prec})$ forms a Leibniz-dendriform coalgebra if and only if
$(A^{*},$ \ \ \ \ $\succ_{A^{*}},\prec_{A^{*}})$ constitutes a Leibniz-dendriform algebra.
Here, $\succ_{A^{*}}$ and $\prec_{A^{*}}$
are the linear dual of $\Delta_{\succ}$ and $\Delta_{\prec}$ respectively defined by the following
relations for all $x \in A$ and $\zeta, \eta \in A^{*}$:
\begin{align}&\label{Da1}\langle \Delta_{\succ}(x),\zeta\otimes \eta\rangle=\langle x,\zeta\succ_{A^{*}} \eta\rangle
\\&\label{Da2}\langle \Delta_{\prec}(x),\zeta\otimes \eta\rangle=\langle x,\zeta\prec_{A^{*}} \eta\rangle,~\forall~x\in A, ~\zeta,\eta\in A^{*}.
\end{align}
Consequently, a Leibniz-dendriform bialgebra $(A, \succ, \prec, \Delta_{\succ}, \Delta_{\prec})$
can be denoted as $(A, \succ, \prec, A^{*},$ \ \ \ \ $ \succ_{A^{*}}, \prec_{A^{*}})$,
where the Leibniz-dendriform algebra structure on $A^{*}$ corresponds to
the Leibniz-dendriform coalgebra structure on $A$ via the above equations.
\end{rmk}

\begin{thm} \label{Mp2}
Let $(A,\succ_A,\prec_A)$ be a Leibniz-dendriform algebra and $(A,\circ )$ be the associated Leibniz algebra of $(A,\succ_{A},\prec_{A})$.
Suppose that there is a Leibniz-dendriform algebra $(A^{*}, \succ_{A^*},\prec_{A^*})$ which is induced from
a Leibniz-dendriform coalgebra $(A, \Delta_{\succ},\Delta_{\prec})$, whose associated Leibniz algebra is denoted by $(A^{*}, \ast)$. Then $(A, A^*,L_{\succ_{A}}^*,-(L_{\succ_{A}}^*+R_{\prec_{A}}^*),L_{\succ_{A^{*}}}^*,-(L_{\succ_{A^{*}}}^*+R_{\prec_{A^{*}}}^*))$
 is a matched pair of Leibniz algebras if and only if $(A,\succ,\prec,\Delta_{\succ},\Delta_{\prec})$ is a Leibniz-dendriform bialgebra.
\end{thm}

\begin{proof} Firstly, we prove that Eq.~(\ref{Lm1}) $\Longleftrightarrow$ Eq.~(\ref{B1})
in the case of $ l_A=L_{\succ_{A}}^*,~r_{A}=-(L_{\succ_{A}}^*+R_{\prec_{A}}^*),~l_B=L_{\succ_{A^{*}}}^*,~ r_B=-(L_{\succ_{A^{*}}}^*+R_{\prec_{A^{*}}}^*)$.
For all $x,y\in A$ and $a,b\in A^{*}$, we have
\begin{align*}\langle -(L_{\succ_{A}}^*+R_{\prec_{A}}^*)(x)(a\ast b),y\rangle=\langle a\ast b,x\odot y\rangle
=\langle \Delta(x\odot y), a\otimes b\rangle,
\end{align*}
\begin{align*}\langle a\ast (L_{\succ_{A}}^*+R_{\prec_{A}}^*)(x)b,y\rangle
=\langle \Delta(y), a\otimes (L_{\succ_{A}}^*+R_{\prec_{A}}^*)(x)b\rangle
=-\langle (I\otimes L_{\odot}(x))\Delta(y), a\otimes b\rangle,
\end{align*}
\begin{align*}&\langle (L_{\succ_{A}}^*+R_{\prec_{A}}^*)(L_{\succ_{A^{*}}}^*(b)x)a,y\rangle
=-\langle a,(L_{\succ_{A^{*}}}^*(b)x)\odot y\rangle
=\langle R_{\succ_{A}}^*(y)(a),L_{\succ_{A^{*}}}^*(b)x\rangle+
\langle L_{\prec_{A}}^*(y)(a),L_{\succ_{A^{*}}}^*(b)x\rangle
\\=&-\langle b\succ_{A^*} R_{\succ_{A}}^*(y)(a),x\rangle-
\langle b\succ_{A^*} L_{\prec_{A}}^*(y)(a),x\rangle
\\=&\langle ((R_{\succ_{A}}+L_{\prec_{A}})(y)\otimes I)\tau\Delta_{\succ}(x), a\otimes b\rangle
\\=&\langle \tau(I\otimes R_{\odot_{A}})(y))\Delta_{\succ}(x), a\otimes b\rangle
\end{align*}
Analogously,
\begin{align*}\langle -b\ast (L_{\succ_{A}}^*+R_{\prec_{A}}^*)(x)b,y\rangle
=-\langle \tau(I\otimes L_{\odot_A}(x))\Delta(y), a\otimes b\rangle,
\end{align*}
\begin{align*}\langle -(L_{\succ_{A}}^*+R_{\prec_{A}}^*)(L_{\succ_{A^{*}}}^*(a)x)b,y\rangle
=-\langle (I\otimes R_{\odot_{A}})(y))\Delta_{\succ}(x), a\otimes b\rangle.
\end{align*}
Thus, Eq.~(\ref{Lm1}) $\Longleftrightarrow$ Eq.~(\ref{B1}). By the same token,
Eq.~(\ref{Lm2}) $\Longleftrightarrow$ Eq.~(\ref{B2}), Eq.~(\ref{Lm3}) $\Longleftrightarrow$ Eq.~(\ref{B3}),
 Eq.~(\ref{Lm4}) $\Longleftrightarrow$ Eq.~(\ref{B4}),
Eq.~(\ref{Lm5}) $\Longleftrightarrow$ Eq.~(\ref{B5}), Eq.~(\ref{Lm6}) $\Longleftrightarrow$ Eq.~(\ref{B6}).
The proof is completed.
\end{proof}

The following result follows directly from Theorems \ref{Mp1} and \ref{Mp2}.

\begin{thm}
Let $(A,\succ_A,\prec_A)$ be a Leibniz-dendriform algebra and $(A,\circ )$ be the associated Leibniz algebra of $(A,\succ_{A},\prec_{A})$.
Suppose that there is a Leibniz-dendriform algebra $(A^{*}, \succ_{A^*},\succ_{A^*})$ which is induced from
a Leibniz-dendriform coalgebra $(A, \Delta_{\prec}, \Delta_{\succ})$, whose associated Leibniz algebra is denoted by $(A^{*}, \ast)$. Then
the following conditions are equivalent:
\begin{enumerate}
 \item $(A\oplus A^{*},\circ,\omega)$ is a phase space of the Leibniz algebra $(A,\circ)$.
\item $(A\oplus A^{*},A, A^{*},\omega)$ is a Manin triple of Leibniz-dendriform algebras with
the bilinear form given by Eq.~(\ref{C2})
and the isotropic subalgebras are $A$ and $ A^{*}$.
 \item $(A,A^{*}, L_{\circ}^{*},L_{\prec_{A}}^{*}+R_{\succ_{A}}^{*},-L_{\prec_A}^{*},-L_{\circ}^{*}-R_{\circ}^{*},
 L_{\ast}^{*},L_{\prec_{A^{*}}}^{*}+R_{\succ_{A^{*}}}^{*},-L_{\prec_{A^{*}}}^{*},-L_{\ast}^{*}-R_{\ast}^{*})$
 is a matched pair of Leibniz-dendriform algebras.
 \item $(A,A^{*},L_{\succ_A}^{*},-L_{\succ_A}^{*}-R_{\prec_A}^{*},L_{\succ_{A^*}}^{*},-L_{\succ_{A^*}}^{*}-R_{\prec_{A^*}}^{*})$
  is a matched pair of
 Leibniz algebras.
\item  $(A,\succ_A,\prec_A,\Delta_{\succ},\Delta_{\prec})$ is a Leibniz-dendriform bialgebra.
 \end{enumerate}
\end{thm}

Let $(A, \succ,\prec,\Delta_{\prec}, \Delta_{\succ})$ be
 a Leibniz-dendriform bialgebra. Then $(D=A\oplus A^{*},\succ_{D},\prec_{D})$
 is a Leibniz-dendriform algebra, where
 \begin{align}\label{Db1}(x+a)\succ_{D}(y+b)=&x\succ_A y+L_{\ast}^{*}
(a)y+(R_{\succ_{A^*}}^*+L_{\prec_{A^*}}^*)
(b)x\\&+a\succ_{A^*}b+L_{\circ}^{*}(x)b+(R_{\succ_A}^{*}+L_{\prec_A}^{*})(y)a\nonumber
,\end{align}
\begin{align}\label{Db2}
(x+a)\prec_{D}(y+b)=&x\prec_A y-L_{\prec_{A^*}}^{*}
(a)y-(R_{\ast}^{*}+L_{\ast}^{*})(b)x\\&+a\prec_{A^*}b-L_{\prec_A}^{*}
(x)b-(R_{\circ}^{*}+L_{\circ}^{*})(y)a\nonumber,
\end{align}
for all $x,y\in A,a,b\in A^{*}$.
$(D=A\oplus A^{*},\succ_{D},\prec_{D})$ is called the double Leibniz-dendriform algebra.

\subsection{Coboundary Leibniz-dendriform bialgebras and the Leibniz-dendriform Yang-Baxter equation}

\begin{defi}
A Leibniz-dendriform bialgebra $(A,\succ,\prec,\Delta_{\succ,r},\Delta_{\prec,r})$ is called coboundary
 if $\Delta_{\succ,r},\Delta_{\prec,r}$ are defined by the following equations respectively:
 \begin{align}&\label{CB1}
\Delta_{\succ,r}(x)=(L_{\odot}(x)\otimes I-I\otimes R_{\circ}(x))r,
\\&\label{CB2}\Delta_{\prec,r}(x)=(L_{\star}(x)\otimes I-I\otimes R_{\prec}(x))\tau(r),
\end{align}
 for all $x\in A$, where
 and $\circ=\succ+\prec,~
 L_{\star}=L_{\circ}+R_{\circ},~L_{\odot}=L_{\succ}+R_{\prec}$.
\end{defi}
It is direct to show that Eqs.~\eqref{B4}-\eqref{B6} hold automatically if $\Delta_{\succ,r},\Delta_{\prec,r}$
are given by Eqs.~\eqref{CB1}-\eqref{CB2}.

\begin{pro} \label{DB4}Let $(A,\succ,\prec)$ be a Leibniz-dendriform algebra. Assume that
$\Delta_{\succ},\Delta_{\prec}$ defined by Eqs.~\eqref{CB1}-\eqref{CB2}, where $r=\sum_i a_i\otimes b_i\in A\otimes A$. Then
\begin{enumerate}
\item Eq.~\eqref{Ca1} holds if and only if the following equation holds:
\begin{align}&\label{CB3}(L_{\odot}(x)\otimes I \otimes I)(r_{23}\circ r_{13}-r_{12}\odot r_{23}+r_{21}\succ r_{13})
\\&+(I\otimes L_{\odot}(x)\otimes I)(r_{21}\odot r_{13}-r_{13}\circ r_{23}-r_{12}\succ r_{23})
\nonumber\\&+(I\otimes I\otimes R_{\circ}(x))(r_{12}\circ r_{23}-r_{13}\odot r_{12}-r_{13}\star r_{21}+r_{21}\prec r_{23}
+r_{13}\circ r_{23})
\nonumber\\&+\sum_{i}(L_{\odot}(x\odot a_i)\otimes I-I\otimes R_{\prec}(x\odot a_i))(r+\tau(r)) \otimes b_i=0
\nonumber.\end{align}
\item Eq.~\eqref{Ca2} holds if and only if the following equation holds:
\begin{align}\label{CB4}&
(I\otimes I \otimes R_{\prec}(x))(r_{31}\odot r_{12}-r_{12}\circ r_{32}-r_{13}\succ r_{32})
\\&+(L_{\odot}(x)\otimes I \otimes I)(r_{12}\star r_{32}-r_{32}\prec r_{13}-r_{12}\prec r_{31})
\nonumber\\&+(I\otimes L_{\star}(x)\otimes I)(r_{13}\circ r_{32}-r_{12}\odot r_{13}+r_{31}\prec r_{32}-r_{12}\star r_{31}+r_{12}\circ r_{32})
\nonumber\\&+\sum_{i}(L_{\odot}( a_i\prec x)\otimes I\otimes I-I\otimes I\otimes R_{\prec}( a_i\prec x))(\tau\otimes I)[b_i\otimes(r+\tau(r))] =0.
\nonumber\end{align}
\item Eq.~\eqref{Ca3} holds if and only if the following equation holds:
\begin{align}\label{CB5}&(L_{\star}(x)\otimes I \otimes I)(r_{21}\odot r_{23}-r_{23}\circ r_{31}+r_{21}\star r_{32}-
r_{32}\prec r_{31}-r_{21}\circ r_{31})
\\&+(I\otimes L_{\odot}(x)\otimes I)(r_{31}\prec r_{23}-r_{21}\star r_{31}+r_{21}\prec r_{32})
\nonumber\\&+(I\otimes I\otimes R_{\prec}(x))(r_{31}\star r_{21}-r_{21}\prec r_{32}-r_{31}\prec r_{23})
\nonumber\\&+\sum_{i}(I\otimes I\otimes R_{\prec}( a_i\prec x)-I\otimes L_{\odot}( a_i\odot x))[b_i\otimes (r+\tau(r))] =0
\nonumber.\end{align}
\item  Eq.~\eqref{B1} holds if and only if the following equation holds:
\begin{align}&\label{CB6}
 [L_{\odot}(x\odot y)\otimes I- L_{\odot}(x)R_{\prec} (y)\otimes I+I\otimes  L_{\odot}(x)R_{\prec}(y)-I\otimes  R_{\prec}(x\odot y)\\&+ L_{\odot}(x)\otimes L_{\odot}(y) -L_{\odot}(y)\otimes  L_{\odot}(x)] (r+\tau(r))=0\nonumber.
 \end{align}
\item  Eq.~\eqref{B2} holds if and only if the following equation holds:
\begin{align}&\label{CB7}
 [L_{\odot}(x\succ y)\otimes I- L_{\succ}(x)L_{\odot}(y)\otimes I+
 I\otimes  L_{\succ}(x)R_{\prec}(y)-I\otimes  R_{\prec}(x\succ y)\\&+ L_{\succ}(x)\otimes R_{\prec}(y) -L_{\odot}(y)\otimes  L_{\succ}(x)] (r+\tau(r))=0\nonumber.
 \end{align}
\item Eq.~\eqref{B3} holds if and only if the following equation holds:
\begin{align}\label{CB8}
( R_{\prec}(x)\otimes R_{\prec}(y))(r+\tau(r))=0.\end{align}
\end{enumerate}
where $L_{\odot}=L_{\succ}+R_{\prec},R_{\odot}=R_{\succ}+L_{\prec},~L_{\star}=L_{\circ}+R_{\circ},~\circ=\succ+\prec$.
\end{pro}

\begin{proof} Using Eqs.~\eqref{CB1}-\eqref{CB2}, we have
\begin{align*}
&( \Delta\otimes I)\Delta_{\succ}(x)-(I\otimes\Delta_{\succ})\Delta_{\succ}(x)+(\tau \otimes I)(I\otimes\Delta_{\succ})\Delta_{\succ}(x)
\\
=&\sum_{i,j}(x\odot a_i)\odot a_j\otimes  b_j\otimes  b_i-a_j\otimes  b_j\circ(x\odot a_i)\otimes b_i-a_i\odot a_j\otimes  b_j\otimes  b_i\circ x
+a_j\otimes b_j\circ a_i\otimes  b_i\circ x\\&+
(x\odot a_i)\star b_j\otimes  a_j\otimes  b_i-b_j\otimes a_j\prec(x\odot a_i)\otimes  b_i-a_i\star b_j\otimes a_j\otimes b_i\circ x
+b_j\otimes a_j\prec a_i\otimes b_i\circ x\\&+
a_i\otimes (b_i\circ x)\odot a_j\otimes b_j-a_i\otimes a_j\otimes b_j\circ(b_i\circ x)-x\odot a_i\otimes b_i\odot a_j\otimes  b_j+
x\odot a_i\otimes  a_j\otimes b_j\circ b_i\\&+a_j\otimes  a_i\otimes b_j\circ (b_i\circ x)-(b_i\circ x)\odot a_j\otimes  a_i\otimes b_j+
b_i\odot a_j\otimes x\odot a_i\otimes b_j-a_j\otimes x\odot a_i\otimes b_j\circ b_i
\\=&P(1)+P(2)+P(3),
\end{align*}
where
\begin{align*}
P(1)=&\sum_{i,j}(x\odot a_i)\odot a_j\otimes  b_j\otimes  b_i+
(x\odot a_i)\star b_j\otimes  a_j\otimes  b_i-x\odot a_i\otimes b_i\odot a_j\otimes  b_j\\&+
x\odot a_i\otimes  a_j\otimes b_j\circ b_i-(b_i\circ x)\odot a_j\otimes  a_i\otimes b_j
,\end{align*}
\begin{align*} P(2)=&
\sum_{i,j}
-a_j\otimes  b_j\circ(x\odot a_i)\otimes b_i-b_j\otimes a_j\prec(x\odot a_i)\otimes  b_i+
a_i\otimes (b_i\circ x)\odot a_j\otimes b_j\\&+b_i\odot a_j\otimes x\odot a_i\otimes b_j-a_j\otimes x\odot a_i\otimes b_j\circ b_i
,\end{align*}
\begin{align*} P(3)=&
\sum_{i,j}-a_i\odot a_j\otimes  b_j\otimes  b_i\circ x
+a_j\otimes b_j\circ a_i\otimes  b_i\circ x-a_i\star b_j\otimes a_j\otimes b_i\circ x
\\&+b_j\otimes a_j\prec a_i\otimes b_i\circ x-a_i\otimes a_j\otimes b_j\circ(b_i\circ x)+a_j\otimes  a_i\otimes b_j\circ (b_i\circ x)
,\end{align*}
By Eq.~(\ref{Ld5}), we obtain
\begin{align*}
P(1)=&\sum_{i,j}(L_{\odot}(x\odot a_i)\otimes I\otimes I)[(r+\tau(r))\otimes b_i]+
b_j\odot(x\odot a_i)\otimes  a_j\otimes  b_i\\&-(L_{\odot}(x)\otimes I \otimes I)(r_{12}\odot r_{23} - r_{23}\circ  r_{13})
-(b_j\circ x)\odot a_i\otimes  a_j\otimes b_i
\\=&\sum_{i,j}(L_{\odot}(x\odot a_i)\otimes I\otimes I)[(r+\tau(r))\otimes b_i]\\&-(L_{\odot}(x)\otimes I \otimes I)(r_{12}\odot r_{23} - r_{23}\circ  r_{13}-r_{21}\succ r_{13}),\end{align*}
\begin{align*} P(2)=&\sum_{i,j}-a_j\otimes  b_j\succ(x\odot a_i)\otimes b_i-(I\otimes R_{\prec}(x\odot a_i)\otimes I)[(r+\tau(r))\otimes b_i]
\\&+a_j\otimes (b_j\circ x)\odot a_i\otimes b_i+(I\otimes L_{\odot}(x)\otimes I) (r_{21}\odot r_{13}-r_{13}\circ  r_{23})
\\=&\sum_{i,j}(I\otimes L_{\odot}(x)\otimes I) (r_{21}\odot r_{13}-r_{13}\circ  r_{23}-r_{12}\succ r_{23})-(I\otimes R_{\prec}(x\odot a_i)\otimes I)[(r+\tau(r))\otimes b_i],\end{align*}
\begin{align*} P(3)=&
(I\otimes I\otimes R_{\circ}(x))(-r_{13}\odot r_{12}+r_{12}\circ r_{23}-r_{13}\star r_{21}+r_{21}\prec r_{23})+
a_i\otimes a_j\otimes (b_i\circ b_j)\circ x
\\=&(I\otimes I\otimes R_{\circ}(x))(-r_{13}\odot r_{12}+r_{12}\circ r_{23}-r_{13}\star r_{21}+r_{21}\prec r_{23}+r_{13}\circ  r_{23})
.\end{align*}
Therefore, Eq.~\eqref{Ca1} holds if and only if Eq.~\eqref{CB3} holds.
An analogous argument proves the remaining part.
\end{proof}

\begin{thm} \label{Y1} Let $(A,\succ,\prec)$ be a Leibniz-dendriform algebra.
Assume that
$\Delta_{\succ,r},\Delta_{\prec,r}$ defined by Eqs.~\eqref{CB1}-\eqref{CB2}. Then $(A,\succ,\prec,\Delta_{\succ,r},\Delta_{\prec,r})$
 is a Leibniz-dendriform bialgebra if and only if Eqs.~\eqref{CB3}-\eqref{CB8} hold.\end{thm}
\begin{proof} This follows from Definition \ref{DB1}, Definition \ref{DB2} and Proposition \ref{DB4}.\end{proof}

Given the complex nature of the compatibility conditions in Eqs.~\eqref{CB3}–\eqref{CB8},
we focus on a simpler special case where
$r\in A \otimes A$
is skew-symmetric. Under this assumption, the following result can be readily established.

\begin{pro} \label{Y2} Assume that $r\in A\otimes A$ is skew-symmetric. Then
\begin{enumerate}
		\item Eqs.~\eqref{CB6}-\eqref{CB8} hold automatically.
\item Eq.~\eqref{CB3} holds if and only if
\begin{align*}(L_{\odot}(x)\otimes I \otimes I)S(r)-(I\otimes L_{\odot}(x)\otimes I)S_{1}(r) +(I\otimes I\otimes R_{\circ}(x))S_{1}(r) =0.\end{align*}

\item Eq.~\eqref{CB4} holds if and only if
\begin{align*}
(I\otimes I \otimes R_{\prec}(x))S_{4}(r) +(L_{\odot}(x)\otimes I \otimes I)S_{5}(r)
-(I\otimes L_{\star}(x)\otimes I)S_{4}(r) =0.\end{align*}
\item Eq.~\eqref{CB5} holds if and only if
\begin{align*}(L_{\star}(x)\otimes I \otimes I)S_{3}(r)
-(I\otimes L_{\odot}(x)\otimes I)S_{2}(r)+(I\otimes I\otimes R_{\prec}(x))
S_{2}(r)=0.\end{align*}
\end{enumerate}
where \begin{align*}&S(r)=r_{23}\circ r_{13}-r_{12}\odot r_{23}-r_{12}\succ r_{13},\ \ \
S_{1}(r)=r_{12}\odot r_{13}+r_{13}\circ r_{23}+r_{12}\succ r_{23},\\&
S_{2}(r)=r_{12}\star r_{13}-r_{12}\prec r_{23}+r_{13}\prec r_{23},\ \ \
S_{3}(r)=r_{23}\odot r_{12}+r_{23}\succ r_{13}-r_{12}\circ r_{13},\\&
S_{4}(r) =-r_{13}\odot r_{12}+r_{12}\circ r_{23}+r_{13}\succ r_{23},\ \ \
S_{5}(r) =r_{23}\prec r_{13}+r_{12}\prec r_{13}-r_{12}\star r_{23}.
\end{align*}
\end{pro}

The following statement is clearly.
\begin{rmk} \label{Y3} Assume that $\sigma_{12},\sigma_{13},\sigma_{23},\sigma_{132}:A\otimes A\otimes A\longrightarrow A\otimes A\otimes A$
are maps defined respectively by
\begin{align*}&\sigma_{12}(x\otimes y\otimes z)=y\otimes x\otimes z, \ \ \ \sigma_{13}(x\otimes y\otimes z)=z\otimes y\otimes x,
\\& \sigma_{132}(x\otimes y\otimes z)=z\otimes x\otimes y,~~\forall~x,y,z\in A.
 \end{align*}
If $r$ is skew-symmetric, then
\begin{align*}&S_{1}(r)=\sigma_{12}S(r),\ \ \  S_{2}(r)=S_{1}(r)-S_{4}(r),\ \ \ S_{4}(r)=-\sigma_{132}S(r),\\
&S_{3}(r)=-\sigma_{13}S(r), \ \ \ S_{5}(r)=S(r)-S_{3}(r).\end{align*}
\end{rmk}

\begin{thm} \label{BY}
Let $\Delta_{\succ,r},\Delta_{\prec,r}$ be linear maps given by Eqs.~\eqref{CB1}-\eqref{CB2} respectively
with $r\in A\otimes A$ being skew-symmetric.
Then $(A,\succ,\prec,\Delta_{\succ,r},\Delta_{\prec,r})$ is a Leibniz-dendriform bialgebra if and only if the
following equation holds:
\begin{equation} \label{YE1}S(r)=r_{23}\circ r_{13}-r_{12}\odot r_{23}-r_{12}\succ r_{13}=0.\end{equation}
\end{thm}
\begin{proof}
The conclusion is established by combining Theorem \ref{Y1}, Proposition \ref{Y2} and Remark \ref{Y3}.
\end{proof}

\begin{defi} Let $(A,\succ,\prec)$ be a Leibniz-dendriform algebra and
$r\in A\otimes A$. Eq.~\eqref{YE1} is called the {\bf Leibniz-dendriform Yang-Baxter equation} or {\bf LD-YBE} in short.
\end{defi}

\subsection{$\mathcal O$-operators and Leibniz-quadri-algebras }
In this section, we introduce the notions of Leibniz-quadri-algebras and $\mathcal{O}$-operators on Leibniz-dendriform algebras.
 We demonstrate that an $\mathcal{O}$-operator on a Leibniz-dendriform algebra gives rise to a Leibniz-quadri-algebra,
and conversely, every Leibniz-quadri-algebra naturally induces an $\mathcal{O}$-operator on its
underlying sub-adjacent Leibniz-dendriform algebra. As an application, we employ $\mathcal{O}$-operators and
 Leibniz-quadri-algebras to construct skew-symmetric solutions of the LD-YBE.

For a vector space $A$, the isomorphism $A\otimes A^{*}\simeq Hom (A^{*},A)$ identifies an element $r\in A\otimes A$ with a map
$T_{r}:A^{*}\longrightarrow A$. Explicitly,
\begin{equation} \label{YE2} T_{r}:A^{*}\longrightarrow A,\ \ \  \langle T_{r}(\zeta),\eta\rangle=\langle r,\zeta\otimes\eta\rangle,
\ \ \ \forall~\zeta,\eta\in A^{*}.\end{equation}
It is clear that $T_{r}^{*}=T_{\tau(r)},~T_{r+\tau(r)}^{*}=T_{r+\tau(r)}$.

\begin{pro} \label{Y4} Let $(A,\succ,\prec)$ be a Leibniz-dendriform
 algebra and
$r=\sum_{i}a_i\otimes b_i\in A\otimes A$. Then the following conclusions hold:
\begin{enumerate}
		\item $r$ is a solution of the LD-YBE: $S(r)=r_{23}\circ r_{13}-r_{12}\odot r_{23}-r_{12}\succ r_{13}=0$
if and only if
\begin{equation*}T_{\tau(r)}(\zeta)\succ T_{\tau(r)}(\eta)=-T_{\tau(r)}(L_{\circ}^*(T_{r}(\zeta))\eta-R_{\odot}^*(T_{\tau(r)}(\eta))\zeta).\end{equation*}
\item $r$ is a solution of the LD-YBE: $S(r)=r_{23}\circ r_{13}-r_{12}\odot r_{23}-r_{12}\succ r_{13}=0$
if and only if
\begin{equation*}T_{r}(\eta)\circ T_{r}(\zeta)=-T_{r}(L_{\odot}^*(T_{r}(\zeta))\eta+L_{\succ}^*(T_{\tau(r)}(\eta))\zeta).\end{equation*}
\item $r$ is a solution of the equation:
$S_{2}(r)=r_{12}\star r_{13}-r_{12}\prec r_{23}+r_{13}\prec r_{23}=0$
if and only if
\begin{equation*}T_{r}(\zeta)\prec T_{r}(\eta)=T_{r}(-L_{\prec}^*(T_{r}(\zeta))\eta+L_{\star}^*(T_{\tau(r)}(\eta))\zeta).\end{equation*}
\item $r$ is a solution of the equation
$S_{3}(r)=r_{23}\odot r_{12}+r_{23}\succ r_{13}-r_{12}\circ r_{13}=0$
if and only if
\begin{equation*}T_{r}(\zeta)\succ T_{r}(\eta)=T_{r}(R_{\odot}^*(T_{r}(\eta))\zeta-L_{\circ}^*(T_{\tau(r)}(\zeta))\eta).\end{equation*}
\item $r$ is a solution of the equation
$S_{3}(r)=r_{23}\odot r_{12}+r_{23}\succ r_{13}-r_{12}\circ r_{13}=0$
if and only if
\begin{equation*}T_{\tau(r)}(\zeta)\circ T_{\tau(r)}(\eta)=-T_{\tau(r)}(L_{\succ}^*(T_{r}(\zeta))\eta+L_{\odot}^*(T_{\tau(r)}(\eta))\zeta).\end{equation*}
\item $r$ is a solution of the equation:
$S_{5}(r) =r_{23}\prec r_{13}+r_{12}\prec r_{13}-r_{12}\star r_{23}=0$
if and only if
\begin{equation*}T_{\tau(r)}(\zeta)\prec T_{\tau(r)}(\eta)=T_{\tau(r)}(L_{\prec}^*(T_{r}(\zeta))\eta-L_{\star}^*(T_{\tau(r)}(\eta))\zeta).\end{equation*}
\end{enumerate}
\end{pro}

\begin{proof} According to Eq.~\eqref{YE2}, for all $\zeta,\eta,\theta\in A^{*}$, we have
\begin{align*}\langle\theta\otimes \zeta\otimes \eta,r_{12}\succ r_{13}\rangle
&=\sum_{i,j}\langle \theta\otimes \zeta\otimes \eta,a_i\succ a_j\otimes b_i\otimes b_j\rangle
\\&=\sum_{i,j}\langle \theta,a_i\succ a_j\rangle\langle \zeta,b_i\rangle \langle \eta,b_j\rangle
=\langle T_{\tau(r)}(\zeta)\succ T_{\tau(r)}(\eta),\theta \rangle
,\end{align*}
\begin{align*}\langle \theta\otimes \zeta\otimes \eta,r_{23}\circ r_{13}\rangle
&=\sum_{i,j}\langle \theta\otimes \zeta\otimes \eta,a_i \otimes a_j\otimes (b_j\circ b_i)\rangle
\\&=\sum_{i,j}\langle \zeta,a_j\rangle \langle \theta,a_i\rangle\langle \eta,b_j\circ b_i\rangle
=\langle \eta,T_{r}(\zeta) \circ T_{r}(\theta)\rangle
\\&=-\langle L_{\circ}^{*}(T_{r}(\zeta))\eta, T_{r}(\theta)\rangle
=-\langle T_{\tau(r)}(L_{\circ}^{*}(T_{r}(\zeta))\eta), \theta\rangle,\end{align*}
 \begin{align*}\langle \theta\otimes \zeta\otimes \eta,r_{12}\odot r_{23}\rangle
 &=\sum_{i,j}\langle \theta\otimes \zeta\otimes \eta,a_i \otimes (b_i\odot a_j)\otimes b_j\rangle
 \\& =\sum_{i,j}\langle \theta,a_i\rangle\langle \eta,b_j\rangle \langle \zeta,b_i\odot a_j\rangle
 =\langle \zeta,T_{r}(\theta)\odot T_{\tau(r)}(\eta)\rangle
  \\&=-\langle R_{\odot}^{*}(T_{\tau(r)}(\eta))\zeta,T_{r}(\theta)\rangle
 =-\langle T_{\tau(r)}(R_{\odot}^{*}(T_{\tau(r)}(\eta))\zeta),\theta\rangle.\end{align*}
Item (a) is thus verified, and the other items follow analogously.
\end{proof}

Recall that an $\mathcal O$-operator $T$ on a Leibniz
algebra $(A,\circ)$ associated to a representation $(V,l,r)$ is a linear map $T:V\longrightarrow A$ satisfying
$T(u)\circ T(v)=T (l(T(u))v+r(T(v))u)$ for all $u,v\in V$.

 \begin{defi} Let $(A,\succ,\prec)$ be a Leibniz-dendriform algebra and $(V,l_{\succ},r_{\succ},l_{\prec},r_{\prec})$ be
 its representation.
  An $\mathcal O$-operator $T$ on $(A,\succ,\prec)$ associated to $(V,l_{\succ},r_{\succ},l_{\prec},r_{\prec})$
   is a linear map $T:V\longrightarrow A$ satisfying
\begin{equation*}T(u)\succ T(v)=T (l_{\succ}(T(u))v+r_{\succ}(T(v))u),\ \ \
T(u)\prec T(v)=T (l_{\prec}(T(u))v+r_{\prec}(T(v))u),~~\forall~u,v\in V.\end{equation*}
In particular, an $\mathcal{O}$-operator $P$ on $(A,\circ )$ associated to
the representation $(A,L_{\succ},R_{\succ},L_{\prec},R_{\prec})$ is called a Rota-Baxter operator, that is, $P:A\longrightarrow A$
is a linear map satisfying
\begin{equation}
P(x)\succ P(y)=P(P(x)\succ y)+x\succ P(y)),\\ \ P(x)\prec P(y)=P(P(x)\prec y)+x\prec P(y)). \end{equation}
More generally, a linear map $P:A\longrightarrow A$ is called a Rota-Baxter of weight $\lambda$ on a Leibniz-dendriform algebra
$(A,\succ,\prec)$
if
\begin{align*} &P (x) \prec P(y)=P(P(x)\prec y+x\prec P(y)+\lambda x\prec y), \\
&P (x) \succ P(y)=P(P(x)\succ y+x\succ P(y)+\lambda x\succ y),\end{align*}
for all $x, y, z \in A$.
$(A,\succ,\prec,P)$ is called a Rota-Baxter Leibniz-dendriform algebra of weight $\lambda$.\end{defi}

Direct calculation shows that
$(A,\succ,\prec,P)$ is a Rota-Baxter Leibniz-dendriform algebra of weight $\lambda$
if and only if $(A,\succ,\prec,\tilde{P})$ is, where $\tilde{P}=-\lambda I -P$.
	
\begin{ex} \label{Eo} Let $(A,\succ,\prec)$ be the 2-dimensional Leibniz-dendriform algebra given in Example \ref{E2}
with a basis $\{e_1, e_2\}$. Define a linear map $T:A\longrightarrow A$ by a matrix
 $\begin {bmatrix}
a&b\\
c&d
\end {bmatrix}$
with respect to the basis $\{e_1, e_2 \}$, where $a,b,c,d\in \mathbb{R}$.
Then $T$ is an $\mathcal{O}$-operator on $(A,\succ,\prec)$ associated to the representation
$(A,L_{\succ},R_{\succ},L_{\prec},R_{\prec})$ if and only if
 $a=d=0, \ bc=0.$
\end{ex}

\begin{thm} \label{YE3} Let $(A,\succ,\prec)$ be a Leibniz-dendriform
 algebra,
$r=\sum_{i}a_i\otimes b_i\in A\otimes A$ be skew-symmetric and $(A,\circ)$ be the associated Leibniz
 algebra of $(A,\succ,\prec)$. Then the following conditions are equivalent:
\begin{enumerate}
		\item $r$ is a solution of the LD-YBE in $(A,\succ,\prec)$, that is,
\begin{equation*} S(r)=r_{23}\circ r_{13}-r_{12}\odot r_{23}-r_{12}\succ r_{13}=0.\end{equation*}
\item $T_{r}$ is an $\mathcal O$-operator on $(A, \circ)$ associated
to $(A^{*},L_{\succ}^{*},-L_{\odot}^{*})$.
\item $T_{r}$ is an $\mathcal O$-operator on $(A,\succ,\prec)$ associated to $(L_{\circ}^*,R_{\odot}^*,-L_{\prec}^*,-L_{\star}^*)$.
\end{enumerate}
\end{thm}

\begin{proof} The statement is deduced by combining Remark \ref{Y3} with Proposition \ref{Y4}.
\end{proof}

Inspired by the theory of quadri-algebras \cite{2}, we are led to introduce the notion of Leibniz-quadri-algebras.

\begin{defi} 
 A {\bf Leibniz-quadri-algebra} is a vector space $A$ together with four binary operations
$\searrow,\nearrow,\swarrow,\nwarrow: A\otimes A \rightarrow A$ satisfying
\begin{align} \label{Lq1}
&(x\circ y)\searrow z=x\searrow(y\searrow z)-y\searrow(x\searrow z),\\
\label{Lq2}&(x\vee y)\nearrow z=x\searrow(y\nearrow z)-y\nearrow (x\succ z),\\
\label{Lq3}&(x\wedge y)\nearrow z=x\nearrow(y\succ z)-y\searrow (x\nearrow z),\\
\label{Lq4}&(x\succ y)\swarrow z+y\swarrow(x\vee z)=x\searrow(y\swarrow z),\\
\label{Lq5}&(x\searrow y)\nwarrow z+y\nwarrow(x\circ z)=x\searrow(y\nwarrow z),\\
\label{Lq6}&(y\nearrow x)\nwarrow z+x\swarrow(y\wedge z)=x\nearrow(y\prec z),\\
\label{Lq7}&x\swarrow(y\vee z)=(x\prec y)\swarrow z+y\searrow(x\swarrow z),\\
\label{Lq8}&x\swarrow(y\wedge z)=(x\swarrow y)\nwarrow z+y\nearrow(x\prec z),\\
\label{Lq9}&x\nwarrow(y\circ z)=(x\nwarrow y)\nwarrow z+y\searrow(x\nwarrow z),
\end{align}
for all $x,y,z\in A$, where $x\succ y=x\nearrow y+x\searrow y,~x\prec y=x\nwarrow y+x\swarrow y,~
x\circ y=x\succ y+x\prec y,~
x\vee y=x\searrow y+x\swarrow y,~x\wedge y=x\nearrow y+x\nwarrow y$.
\end{defi}

The validity of Eq.~(\ref{Ld1}) for the pair $(\succ,\prec)$ is established by Eqs.~(\ref{Lq1})-(\ref{Lq3}). 
Likewise, Eqs.~(\ref{Lq4})-(\ref{Lq6}) establish Eq.~(\ref{Ld2}), 
and Eqs.~(\ref{Lq7})-(\ref{Lq9}) establish Eq.~(\ref{Ld3}) for this pair.
Moreover, the following combinations establish the corresponding relations for $(\vee,\wedge)$
Eq.~(\ref{Ld1}) is derived from Eqs.~(\ref{Lq1}), (\ref{Lq4}) and (\ref{Lq7}),
Eq.~(\ref{Ld2}) is derived from Eqs.~(\ref{Lq2}), (\ref{Lq5}) and (\ref{Lq8}),
Eq.~(\ref{Ld3}) is derived from Eqs.~(\ref{Lq3}), (\ref{Lq6}) and (\ref{Lq9}).
Thus, we have

\begin{rmk} \begin{enumerate}
		\item $(A,\succ,\prec)$ is a Leibniz-dendriform algebra. We denote it by $A_{h}$ and call
it the horizontal Leibniz-dendriform algebra associated to $(A,\searrow,\nearrow,\swarrow,\nwarrow)$.
\item $(A,\vee,\wedge)$ is a Leibniz-dendriform algebra. We denote it by $A_{v}$ and call
it the vertical Leibniz-dendriform algebra associated to $(A,\searrow,\nearrow,\swarrow,\nwarrow)$.
\item $(A,\circ)$ is a Leibniz algebra 
associated to both the horizontal Leibniz-dendriform algebra $(A_{h},\succ,\prec)$ and the vertical Leibniz-dendriform algebra $(A_{v},\vee,\wedge)$.
\end{enumerate}
\end{rmk}

\begin{pro} \label{Qq1}Let $A$ be a vector space. Suppose that $\searrow,\nearrow,\swarrow,\nwarrow: A\otimes A \rightarrow A$
  are binary operations. Then the following are equivalent.
  \begin{enumerate}
\item $(A,\searrow,\nearrow,\swarrow,\nwarrow)$ is a Leibniz-quadri-algebra 
\item $(A,\succ,\prec)$ is a Leibniz-dendriform algebra
and $(L_{\searrow},R_{\nearrow},L_{\swarrow},R_{\nwarrow})$ is a representation of $(A,\succ,\prec)$
\item $(A,\vee,\wedge)$ is a Leibniz-dendriform algebra
and $(L_{\searrow},R_{\swarrow},L_{\nearrow},R_{\nwarrow})$ is a representation of $(A,\vee,\wedge)$.
\end{enumerate}
\end{pro}
\begin{proof}
It can be checked by direct calculation.
\end{proof}

\begin{ex} \label{Qq2}Let $(A,\searrow,\nearrow,\swarrow,\nwarrow)$ be a Leibniz-quadri-algebra. Then 
\begin{enumerate}
		\item The identity map
$I$ is an $\mathcal{O}$-operator $T$ on Leibniz-dendriform algebra $(A_{h},\succ,\prec)$ associated to the representation
$(A,L_{\searrow},R_{\nearrow},L_{\swarrow},R_{\nwarrow})$.
\item The identity map
$I$ is an $\mathcal{O}$-operator $T$ on Leibniz-dendriform algebra $(A_{v},\vee,\wedge)$ associated to the representation
$(A,L_{\searrow},R_{\swarrow},L_{\nearrow},R_{\nwarrow})$.
\end{enumerate}
\end{ex}

\begin{pro} Let $(A,\succ_1,\prec_1)$ and $(B,\succ_2,\prec_2)$ be two Leibniz-dendriform algebras.
Define operations $\searrow,\nearrow,\swarrow,\nwarrow: (A\otimes B)\otimes A\otimes B \rightarrow A\otimes B$ by
\begin{align*} 
&(x\otimes a)\nwarrow (y\otimes b)=(x\prec_1 y)\otimes (a\prec_2 b),\\
&(x\otimes a)\swarrow (y\otimes b)=(x\prec_1 y)\otimes (a\succ_2 b),\\
&(x\otimes a)\nearrow (y\otimes b)=(x\succ_1 y)\otimes (a\prec_2 b),\\
&(x\otimes a)\searrow (y\otimes b)=(x\succ_1 y)\otimes (a\succ_2 b)
\end{align*}
for all $x,y\in A$ and $a,b\in B$. Then $(A\otimes B,\searrow,\nearrow,\swarrow,\nwarrow)$ is a
Leibniz-quadri-algebra. The corresponding horizontal Leibniz-dendriform algebra $(A\otimes B)_{h}$ structure
$(\succ,\prec)$ is given by
\begin{align*} 
&(x\otimes a)\prec (y\otimes b)=(x\prec_1 y)\otimes (a\circ_2 b),\\
&(x\otimes a)\succ (y\otimes b)=(x\succ_1 y)\otimes (a\circ_2 b)
\end{align*}
and the corresponding vertical Leibniz-dendriform algebra $(A\otimes B)_{v}$ structure
$(\wedge,\vee)$ is given by
\begin{align*} 
&(x\otimes a)\vee (y\otimes b)=(x\circ_1 y)\otimes (a\succ_2 b),\\
&(x\otimes a)\wedge (y\otimes b)=(x\circ_1 y)\otimes (a\prec_2 b).
\end{align*}
\end{pro}

\begin{proof} For all $x,y,z\in A$ and $a,b,\in B$, we have 
\begin{align*}&((x\otimes a)\circ (y\otimes b))\searrow (z\otimes c)=(x\circ_1 y)\otimes (a\circ_2 b)\searrow (z\otimes c)
\\=&((x\circ_1 y)\succ_1 z)\otimes ((a\circ_2 b)\succ_2 c),\\
& (x\otimes a)\searrow((y\otimes b)\searrow (z\otimes c))-(y\otimes b)\searrow((x\otimes a)\searrow (z\otimes c))
\\=&(x\otimes a)\searrow ((y\succ_1 z)\otimes (b\succ_2 c))-(y\otimes b)\searrow ((x\succ_1 z)\otimes (a\succ_2 c))
\\=&(x\succ_1(y\succ_1 z))\otimes (a\succ_2(b\succ_2 c))-(y\succ_1(x\succ_1 z))\otimes (b\succ_2(a\succ_2 c)).
\end{align*}
Combining Eq.~(\ref{Ld1}), we get that Eq.~(\ref{Lq1}) hold. Analogously,
Eqs.~(\ref{Lq2})-(\ref{Lq9}) hold.
\end{proof}

\begin{pro}\label{Lq} Let $(A,\succ,\prec)$ be a Leibniz-dendriform algebra and $T:V\longrightarrow A$ an $\mathcal O$-operator
on $(A,\succ,\prec)$ associated to the representation $(V,l_{\succ},r_{\succ},l_{\prec},r_{\prec})$.
Define four operations $\searrow_{T},\nearrow_{T},\swarrow_{T},\nwarrow_{T}: V\otimes V \rightarrow V$ by
\begin{align*} 
&u\searrow_{T} v=l_{\succ}(T(u))v,\ \  \ u\nearrow_{T} v=r_{\succ}(T(v))u,\\
&u\swarrow_{T} v=l_{\prec}(T(u))v,\ \  \ u\nwarrow_{T} v=r_{\prec}(T(v))u
\end{align*}
for all $u,v\in V$. Then $(V,\searrow_{T},\nearrow_{T},\swarrow_{T},\nwarrow_{T})$ is a
Leibniz-quadri-algebra. The corresponding horizontal Leibniz-dendriform algebra structure
$(\succ_{T},\prec_{T})$ is given by 
\begin{align*} 
u\prec_{T} v=l_{\prec}(T(u))v+r_{\prec}(T(v))u, \ \  \ u\succ_{T} v=l_{\succ}(T(u))v+r_{\succ}(T(v))u.
\end{align*}
and the corresponding vertical Leibniz-dendriform algebra structure
$(\vee_{T},\wedge_{T})$ is given by 
\begin{align*} 
u\vee_{T} v=l_{\circ}(T(u))v, \ \  \ u\wedge_{T} v=r_{\circ}(T(v))u.
\end{align*}
Moreover, $T$ is both a Leibniz-dendriform algebra homomorphism from $(V_{h},\succ_{T},\prec_{T})$ to $(A,\succ,\prec)$
and a Leibniz algebra homomorphism from $(V,\circ_{T})$ to $(A,\circ)$.
\end{pro}
\begin{proof}
It can be checked by direct computations.
\end{proof}

\begin{cor}\label{Qp} Let $(A,\succ,\prec)$ be a Leibniz-dendriform algebra and $T:V\longrightarrow A$ an $\mathcal O$-operator
on $(A,\succ,\prec)$ associated to the representation $(V,l_{\succ},r_{\succ},l_{\prec},r_{\prec})$.
Then $T(V)=\{T(v)|v\in V\}\subset A$ is a subalgebra of $A$ and there is an induced Leibniz-quadri-algebra structure
on $T(V)$ given by
\begin{align}\label{Qr01} 
&T(u)\swarrow T( v)=T(u\swarrow v)=T(l_{\prec}(T(u))v),\ \  \ T(u)\nwarrow T( v)=T(u\nwarrow v)=T(r_{\prec}(T(v))u),\\
\label{Qr02}&T(u)\searrow T( v)=T(u\searrow v)=T(l_{\succ}(T(u))v),\ \  \ T(u)\nearrow T( v)=T(u\nearrow v)=T(r_{\succ}(T(v))u)
\end{align}
for all $u,~v\in V.$
\end{cor}

By Theorem \ref{YE3} and Proposition \ref{Lq}, we have
 \begin{ex} Let $(A,\succ,\prec)$ be a Leibniz-dendriform algebra and $r\in A\otimes A$ a skew-symmetric solution of the LD-YBE.
Then $(A^{*},\searrow,\nearrow,\swarrow,\nwarrow)$ is a Leibniz-quadri-algebra, where
\begin{align*} 
&a\swarrow b=-L_{\prec}^{*}(T_{r}(a))b,\ \  \ a\nwarrow b=-(L_{\circ}^{*}+R_{\circ}^{*})(T_{r}(b))a,\\
&a\searrow b=L_{\circ}^{*}(T_{r}(a))b,\ \  \ a\nearrow b=R_{\odot}^{*}(T_{r}(b))a,~ \forall a,~b\in A^{*}.
\end{align*}
\end{ex}

\begin{defi} 
 Let $(A,\succ,\prec)$ be a Leibniz-dendriform algebra and $\omega$ be skew-symmetric bilinear map.
 $\omega$ is called a
2-cocycle on $(A,\succ,\prec)$ if $\omega$ satisfies
\begin{align} 
\label{Qr1} \omega(x\succ y+x\prec y,z)=\omega(x,y\succ z+z\prec y)-\omega(y,x\succ z),~~\forall~x,y,z\in A.
\end{align}
\end{defi}

\begin{pro} 
 Let $(A,\succ,\prec)$ be a Leibniz-dendriform algebra together with an non-degenerate
2-cocycle $\omega$, then
 there is a compatible Leibniz-quadri-algebra structure $(\searrow,\nearrow,\swarrow,\nwarrow)$ on $A$
 given by Eqs.~(\ref{Qr2})-(\ref{Qr3}) such that $(A,\succ,\prec)$ is the associated Leibniz-dendriform algebra.
Conversely, suppose that $(A,\searrow,\nearrow,\swarrow,\nwarrow)$ is a Leibniz-quadri-algebra equipped with 
a skew-symmetric bilinear map $\omega$. If $\omega$ is invariant on $A$, that is,
\begin{align} \label{Qr2}
&\omega(x\searrow y,z)=-\omega(y,x\circ z), \ \ \ \omega(x\swarrow y,z)=\omega(y,x\prec z),\\
\label{Qr3} &\omega(x\nearrow y,z)=-\omega(x,z\odot y), \ \ \ \omega(x\nwarrow y,z)=\omega(x,y\star z),
\end{align}
 where $x\circ y=x\succ y+x\prec y,~x\odot y=x\succ y+y\prec x$ and $x\star y=x\circ y+y\circ x$.
 Then $\omega$ is a 2-cocycle of the associated Leibniz-dendriform algebra $(A_{h},\succ,\prec)$.
\end{pro}

\begin{proof} 
Since $\omega$ is a non-degenerate skew-symmetric bilinear map on $A$, which can induce an invertible linear map
$T:A^{*}\longrightarrow A$ by
\begin{align*} 
\omega(T(a),y)=\langle a,y\rangle, \ \ \ \forall~a\in A^{*},~y\in A.
\end{align*}
By Eq.~(\ref{Qr1}), we have for all $a,b,c\in A^{*}$, 
\begin{align*} 
&\langle c,T(a)\succ T(b)\rangle=\omega(T(c),T(a)\succ T(b))\\=&
\omega(T(b),T(a)\circ T(c) )+\omega(T(a), T(c)\odot T(b))
\\=&\langle a,T(c)\odot T(b)\rangle +\langle b,T(a)\circ T(c)\rangle
\\=&-\langle L_{\circ}^{*}(T(a))b, T(c)\rangle-\langle R_{\odot}^{*}(T(b))a,T(c)\rangle
\\=&-\omega (T(L_{\circ}^{*}(T(a))b),T(c))-\omega (T(R_{\odot}^{*}(T(b))a),T(c))
\\=&\omega (T(c),T(L_{\circ}^{*}(T(a))b)) +\omega (T(c),T(R_{\odot}^{*}(T(b))a))
\\=&\langle c,T(L_{\circ}^{*}(T(a))b) +T(R_{\odot}^{*}(T(b))a)\rangle,
\end{align*}
which implies that 
  \begin{align*}T(a)\succ T(b)= T(L_{\circ}^{*}(T(a))b) +T(R_{\odot}^{*}(T(b))a).\end{align*}
 Analogously,
 \begin{align*}T(a)\prec T(b)= T(-L_{\prec}^{*}(T(a))b-L_{\star}^{*}(T(b))a).\end{align*}
 Thus, $T$ is an $\mathcal{O}$-operator on $(A,\succ,\prec)$ associated to the representation
$(A,L_{\circ}^{*},R_{\odot}^{*},-L_{\prec}^{*},-L_{\star}^{*})$.
Combining Proposition \ref{Lq}, there is a compatible Leibniz-quadri-algebra structure on $A$
given as follows:
\begin{align*} 
&a\swarrow b=-L_{\prec}^{*}(T(a))b,\ \  \ a\nwarrow b=-(L_{\circ}^{*}+R_{\circ}^{*})(T(b))a,\\
&a\searrow b=L_{\circ}^{*}(T(a))b,\ \  \ a\nearrow b=R_{\odot}^{*}(T(b))a,~ \forall a,~b\in A^{*}.
\end{align*}
For all $x,y\in A$, there are $a,b\in A^{*}$ such that $T(a)=x,~T(b)=y$.
In light of Corollary \ref{Qp}, there is a Leibniz-quadri-algebra structure on $A$
given by Eqs.~(\ref{Qr01})-(\ref{Qr02}). In detail,
\begin{align*} 
&x\swarrow y=T(-L_{\prec}^{*}(T(a))b),\ \  \ x\nwarrow y=-T((L_{\circ}^{*}+R_{\circ}^{*})(T(b))a),\\
&x\searrow y=T(L_{\circ}^{*}(T(a))b),\ \  \ x\nearrow y=T(R_{\odot}^{*}(T(b))a),~ \forall a,~b\in A^{*}.
\end{align*}
Thus,
\begin{align*} 
&\omega(x\searrow y,z)=\omega(T(-L_{\prec}^{*}(T(a))b),z)=\langle -L_{\prec}^{*}(T(a))b,z \rangle
\\=&\langle b,T(a)\prec z \rangle=\omega(T(b),T(a)\prec z )
\\=&\omega(y,x\prec z ),\\
&\omega(x\searrow y,z)=\omega(T(L_{\circ}^{*}(T(a))b),z)=
\langle L_{\circ}^{*}(T(a))b,z \rangle
\\=&-\langle b,T(a)\circ z \rangle
=-\omega(T(b),T(a)\circ z)
\\=&-\omega(y,x\circ z).
\end{align*}
By the same token, Eq.~(\ref{Qr3}) holds. The converse is straightforward.
 \end{proof}
 
\begin{thm} \label{Yo}
 Let $(A,\succ,\prec)$ be a Leibniz-dendriform algebra and $(V,l_{\succ},r_{\succ},l_{\prec},r_{\prec})$ be
  a representation of $(A,\succ,\prec)$. Suppose that $(V^{*},l_{\circ}^*,r_{\odot}^*,-l_{\prec}^*,
-l_{\star}^*)$ is the dual representation of $A$ given by Proposition \ref{Dr}.
Let $\hat{A}=A\ltimes V^{*}$ and
 $T:V\longrightarrow A$ be a linear map which is identifies an element in $\hat{A}\otimes \hat{A}$ through
 ($Hom(V,A)\simeq A\otimes V^{*}\subseteq \hat{A}\otimes \hat{A}$).
  Then $r=T-\tau(T)$ is a skew-symmetric solution
of the LD-YBE in the Leibniz-dendriform algebra $\hat{A}$ if and only if $T$ is an $\mathcal O$-operator
on $(A,\succ,\prec)$ associated to $(V,l_{\succ},r_{\succ},l_{\prec},r_{\prec})$.
\end{thm}

\begin{proof}
For all $x+a^{*},y+b^{*}\in \hat{A}$ with $x,y\in A$ and $a^{*},b^{*}\in V^{*}$, the Leibniz-dendriform algebraic structure
$(\succ,\prec)$ on $\hat{A}$ is defined by
\begin{align}\label{YE9}
 &(x+a^{*})\succ(y+b^{*})=x\succ y+l_{\circ}^*(x)b^{*}+r_{\odot}^{*}(y)a^{*},
 \\& \label{YE10}(x+a^{*})\prec(y+b^{*})=x\prec y-l_{\prec}^{*}(x)b^{*}-l_{\star}^*(y)a^{*},\end{align}
and the associated Leibniz algebraic structure $\circ$ on $\hat{A}$ is given by
 \begin{align*}
 (x+a^{*})\circ(y+b^{*})=x\circ y+l_{\succ}^*(x)b^{*}-l_{\odot}^{*}(y)a^{*}.\end{align*}
Assume that $\{v_1,v_2,\cdot\cdot\cdot, v_n\}$ is a basis of $V$ and
$\{v_1^{*},v^{*}_2,\cdot\cdot\cdot, v^{*}_n\}$ is the dual basis of $V^{*}$. Then
$T=\sum_{i=1}^{n}T(v_i)\otimes v^{*}_i\in T(V)\otimes V^{*}\subseteq \hat{A}\otimes \hat{A}$. Note that
 \begin{align}\label{YE11}
 &l_{\succ}^{*}(T(v_{i}))v_{j}^{*}
 =\sum_{k=1}^{n}\langle -v_{j}^{*},l_{\succ}(T(v_{i})v_k\rangle v_{k}^{*}, \ \ \ r_{\succ}^{*}(T(v_{i}))v_{j}^{*}=
 \sum_{k=1}^{n}\langle -v_{j}^{*},r_{\succ}(T(v_{i})v_k\rangle v_{k}^{*},\\&
 \label{YE12} l_{\prec}^{*}(T(v_{i}))v_{j}^{*}=
 \sum_{k=1}^{n}\langle -v_{j}^{*},l_{\prec}(T(v_{i})v_k\rangle v_{k}^{*}
, \ \ \ r_{\prec}^{*}(T(v_{i}))v_{j}^{*}= \sum_{k=1}^{n}\langle -v_{j}^{*},r_{\prec}(T(v_{i})v_k\rangle v_{k}^{*}.\end{align}
 Using Eqs.~\eqref{YE9}-\eqref{YE12}, we compute:
 \begin{align*}
r_{23}\circ r_{13}
 &=\sum_{i,j=1}^{n}-T(v_j)\otimes v_{i}^{*}\otimes (T(v_i)\circ  v_{j}^{*})-
v_{j}^{*}\otimes T(v_{i})\otimes (v_{i}^{*}\circ T(v_j))
+v_{j}^{*}\otimes v_{i}^{*}\otimes (T(v_i)\circ T(v_j))
\\&=\sum_{i,j=1}^{n}-T(v_j)\otimes v_{i}^{*}\otimes [l_{\succ}^{*}(T(v_i))
 v_{j}^{*}]+v_{j}^{*}\otimes T(v_{i})\otimes l_{\odot}^{*}(T(v_j))v_{i}^{*}
+v_{j}^{*}\otimes v_{i}^{*}\otimes (T(v_i)\circ T(v_j))
\\&=\sum_{i,j=1}^{n}T(l_{\succ}(T(v_i))v_j)\otimes v_{i}^{*}\otimes
 v_{j}^{*})-v_{j}^{*}\otimes T(l_{\odot}(T(v_j))v_{i})\otimes v_{i}^{*}
+v_{j}^{*}\otimes v_{i}^{*}\otimes (T(v_i)\circ T(v_j)),\end{align*}
 \begin{align*}
 r_{12}\succ r_{13}
 &=\sum_{i,j=1}^{n}T(v_i)\succ T(v_j)\otimes v_{i}^{*}\otimes v_{j}^{*}-
 T(v_i)\succ v_{j}^{*}\otimes v_{i}^{*}\otimes T(v_{j})
-v_{i}^{*}\succ T(v_j)\otimes T(v_{i})\otimes v_{j}^{*}
\\&=\sum_{i,j=1}^{n}T(v_i)\succ T(v_j)\otimes v_{i}^{*}\otimes v_{j}^{*}
-[l_{\circ}^{*}(T(v_i))v_{j}^{*}]\otimes v_{i}^{*}\otimes T(v_{j})
-r_{\odot}^{*}(T(v_j))v_{i}^{*}\otimes T(v_{i})\otimes v_{j}^{*}
\\&=\sum_{i,j=1}^{n}T(v_i)\succ T(v_j)\otimes v_{i}^{*}\otimes v_{j}^{*}
+v_{j}^{*}\otimes v_{i}^{*}\otimes T(l_{\circ}(T(v_i))v_{j})
+v_{i}^{*}\otimes T(r_{\odot}(T(v_j))v_{i})\otimes v_{j}^{*},\end{align*}
\begin{align*}
r_{12}\odot r_{23}
 &=\sum_{i,j=1}^{n}T(v_i)\otimes (v_{i}^{*}\odot  T(v_j))\otimes v_{j}^{*}+
v_{i}^{*}\otimes (T(v_{i})\odot v_{j}^{*})\otimes T(v_{j})-
v_{i}^{*}\otimes (T(v_{i})\odot T(v_{j}))\otimes v_{j}^{*}
\\&=\sum_{i,j=1}^{n}T(v_i)\otimes v_{i}^{*}\otimes [r_{\succ}^*(T(v_j))v_{j}^{*}]
-v_{i}^{*}\otimes r_{\circ}^*(T(v_{i}))v_{j}^{*} \otimes T(v_{j})
-v_{i}^{*}\otimes (T(v_{i})\odot T(v_{j}))\otimes v_{j}^{*}
\\&=\sum_{i,j=1}^{n}-T(r_{\succ}(T(v_j))v_i)\otimes v_{i}^{*}\otimes v_{j}^{*}
+v_{i}^{*}\otimes v_{j}^{*} \otimes T(r_{\circ}(T(v_{i})v_{j}))
-v_{i}^{*}\otimes (T(v_{i})\odot T(v_{j}))\otimes v_{j}^{*}.\end{align*}
Thus, $r=T-\tau(T)$ is a skew-symmetric solution
of the LD-YBE in the Leibniz-dendriform algebra $\hat{A}$ if and only if the following equations hold:
 \begin{align}\label{YE13}
&T(v_i)\circ T(v_j)=T(T(l_{\circ}(T(v_i))v_j+r_{\circ}(Tv_j)v_i)),
\\&\label{YE14}
 T(v_i)\odot T(v_j)=T(l_{\odot}(T(v_i))v_{j})+T(r_{\odot}(T(v_{j})v_{i}),
\\&\label{YE15}
T(v_{i})\succ T(v_{j})=T(T(l_{\succ}(T(v_i))v_{j}+r_{\succ}(T(v_j))v_{i})).\end{align}
It is straightforward to verify that Eqs.~\eqref{YE13}-\eqref{YE15} hold if and only if
 $T$ is an $\mathcal O$-operator
on $(A,\succ,\prec)$ associated with $(V,l_{\succ},r_{\succ},l_{\prec},r_{\prec})$.
The proof is finished.
\end{proof}

\begin{ex} Let $(A,\succ,\prec)$ be the 2-dimensional Leibniz-dendriform algebra with a basis $\{e_1,e_2\}$
given in Example \ref{E2}. Based on Example \ref{Eo}, we know that
the linear map
\begin{align*}T:A\longrightarrow A ,~T(e_1)=te_2,\ \ T(e_2)=0,\ \ t\in \mathbb{R} \end{align*}
is an $\mathcal{O}$-operator $T$ on $(A,\succ,\prec)$ associated to the representation
$(A,L_{\succ},R_{\succ},L_{\prec},R_{\prec})$.
Denote the dual basis of $A^{*}$ by $\{e_1^{*}, e_2^{*} \}$.
The semi-direct product $A\ltimes A^{*}$ of $(A,\succ,\prec)$ and its representation
$(A^{*},L_{\circ}^{*},R_{\odot}^{*},-L_{\prec}^{*},-L_{\circ}^{*}-R_{\circ}^{*})$ is a
Leibniz-dendriform algebra with the binary operation $(\succ,\prec)$ given by~(only non-trivial operations are listed)
   \begin{flalign*}
&e_1\succ e_1=e_1,\ \ \ e_1\prec e_1=-e_1, \ \ \ e_1\succ e_2=e_2,\ \ \ e_2\prec e_1=-e_2,
\\& e_1\prec e_1^{*}= -e_1^{*}, \ \ \ e_2\prec e_2^{*}= -e_1^{*},\ \ \ e_1\succ e_2^{*}= -e_2^{*}.
\end{flalign*}
 On the basis of Theorem \ref{Yo},
 \begin{align*}r=\sum_{i=1}^{2}T(e_i)\otimes e_i^{*}-e_i^{*}\otimes T(e_i)=te_2\otimes e_1^{*}- te_1^{*}\otimes e_2.
 \end{align*}
 is a skew-symmetric solution of the LD-YBE in the Leibniz-dendriform algebra $(A\ltimes A^{*},\succ,\prec)$. By Theorem \ref{BY},
$ (A\ltimes A^{*},\succ,\prec,\Delta_{\succ,r},\Delta_{\prec,r})$ is a Leibniz-dendriform bialgebra with the linear maps
$\Delta_{\succ,r},\Delta_{\prec,r}:A\ltimes A^{*}\longrightarrow (A\ltimes A^{*})\otimes (A\ltimes A^{*})$ defined respectively by
 \begin{align*}&
\Delta_{\succ,r}(x)=(L_{\odot}(x)\otimes I-I\otimes R_{\circ}(x))r,
\\&\Delta_{\prec,r}(x)=(L_{\star}(x)\otimes I-I\otimes R_{\prec}(x))\tau(r).
\end{align*}
Explicitly,
 \begin{align*}&
\Delta_{\succ,r}(e_1)=-te_1^{*}\otimes e_2, \ \  \Delta_{\succ,r}(e_2^{*})=-2te_1^{*}\otimes e_1^{*},\ \ \
\Delta_{\prec,r}(e_2^{*})=2te_1^{*}\otimes e_1^{*},
\\&\Delta_{\prec,r}(e_i)=\Delta_{\succ,r}(e_2)=\Delta_{\prec,r}(e_1^{*})=\Delta_{\succ,r}(e_1^{*})=0,~(i=1,2).
\end{align*}
\end{ex}

Combining Proposition \ref{Dr}, Proposition \ref{Qq1}, Example \ref{Qq2} and Theorem \ref{Yo}, we have

\begin{cor} Let $(A,\searrow,\nearrow,\swarrow,\nwarrow)$ be a Leibniz-quadri-algebra. Then 
$r=\sum_{i}e_i\otimes e_{i}^{*}-e_{i}^{*}\otimes e_i$ is a skew-symmetric solution of the LD-YBE in
the Leibniz-dendriform algebras $A_{h}\ltimes A^{*}$ and $A_{v}\ltimes A^{*}$,
 where $\{e_1,e_2,\cdot \cdot\cdot,e_n\}$ is a basis of $A$
and $\{e_{1}^{*},e_{2}^{*},\cdot \cdot\cdot,e_{n}^{*}\}$ is the dual basis.
\end{cor}
\section{Quasi-triangular Leibniz-dendriform bialgebras and factorizable Leibniz-dendriform bialgebras}

In light of Theorem \ref{BY}, a skew-symmetric solution of the LD-YBE yields a
 Leibniz-dendriform bialgebra. Here, we extend this construction to
  general solutions, demonstrating that such bialgebras exist beyond the skew-symmetric case.

\subsection{Quasi-triangular Leibniz-dendriform bialgebras }

\begin{defi} \label{In1}
 Let $(A,\succ,\prec)$ be a Leibniz-dendriform algebra and $r\in A\otimes A$. Then $r$ is called {\bf invariant} if
 \begin{align}&\label{Iv1}
(L_{\odot}(x)\otimes I-I\otimes R_{\circ}(x))r=0,
\\&\label{Iv2}(L_{\star}(x)\otimes I-I\otimes R_{\prec}(x))\tau(r)=0,
\end{align}
\end{defi}

 \begin{lem} \label{In2}
 Let $(A,\succ,\prec)$ be a Leibniz-dendriform algebra and $r\in A\otimes A$. Then $r$ is invariant if and only if
 \begin{align}&\label{Iv3}
R_{\circ}(x) T_{r}(\zeta)+T_{r}(L_{\odot}^{*}(x)\zeta)=0,
\\&\label{Iv4}L_{\star}(x) T_{r}(\zeta)+T_{r}(R_{\prec}^{*}(x)\zeta)=0,~~\forall~x\in A,\zeta\in A^{*}.
\end{align}
Moreover, Eqs.~(\ref{Iv3})-(\ref{Iv4}) hold if and only if the following equations hold:
\begin{align}&\label{Iv5}
L_{\circ}^{*} (T_{r}(\zeta))\eta=R_{\odot}^{*}(T_{\tau(r)}(\eta))\zeta),
\\&\label{Iv6}L_{\star}^{*} (T_{r}(\zeta))\eta=L_{\prec}^{*}(T_{\tau(r)}(\eta))\zeta,~~\forall~x\in A,\zeta\in A^{*}.
\end{align}
\end{lem}
\begin{proof}
For all $~x\in A,\zeta,\eta\in A^{*}$, we have
\begin{align*}&\langle (L_{\odot}(x)\otimes I-I\otimes R_{\circ}(x))r,\zeta\otimes \eta\rangle
=\langle r,\zeta\otimes R_{\circ}^{*}(x)\eta\rangle-\langle r,L_{\odot}^{*}(x)\zeta\otimes\eta\rangle
\\=&\langle T_{r}(\zeta),R_{\circ}^{*}(x)\eta\rangle-\langle T_{r}(L_{\odot}^{*}(x)\zeta), \eta\rangle
=-\langle R_{\circ}(x)T_{r}(\zeta)+T_{r}(L_{\odot}^{*}(x)\zeta), \eta\rangle,\end{align*}
\begin{align*}& \langle (L_{\star}(x)\otimes I-I\otimes R_{\prec}(x))\tau(r),\zeta\otimes \eta\rangle
=\langle \tau(r),\zeta\otimes R_{\prec}^{*}(x)\eta\rangle
-\langle \tau(r),L_{\star}^{*}(x)\zeta\otimes \eta\rangle
\\=&\langle \zeta,T_{r}(R_{\prec}^{*}(x)\eta)\rangle-\langle T_{r}(\eta), L_{\star}^{*}(x)\zeta\rangle
=\langle T_{r}(R_{\prec}^{*}(x)\eta),\zeta\rangle+\langle L_{\star}(x)T_{r}(\eta), \zeta\rangle.
\end{align*}
Thus, Eqs.~(\ref{Iv1})-(\ref{Iv2}) hold if and only if Eqs.~(\ref{Iv3})-(\ref{Iv4}) hold. Analogously,
Eqs.~(\ref{Iv3})-(\ref{Iv4}) hold if and only if Eqs.~(\ref{Iv5})-(\ref{Iv6}) hold.
\end{proof}

\begin{pro} \label{In3}
 Let $(A,\succ,\prec)$ be a Leibniz-dendriform algebra and $r\in A\otimes A$. Then the following
 conditions are equivalent:
  \begin{enumerate}
\item $r+\tau(r)$ is invariant.
\item The following equations hold:
\begin{align}&\label{Iv7}
R_{\circ}(x) T_{r+\tau(r)}(\zeta)=-T_{r+\tau(r)}(L_{\odot}^{*}(x)\zeta),\ \ \
L_{\star}(x) T_{r+\tau(r)}(\zeta)=-T_{r+\tau(r)}(R_{\prec}^{*}(x)\zeta).\end{align}
\item The following equations hold:
\begin{align}&\label{Iv8}
L_{\circ}^{*} (T_{r+\tau(r)}(\zeta))\eta=R_{\odot}^{*}(T_{r+\tau(r)}(\eta))\zeta,
\ \ \ L_{\star}^{*} (T_{r+\tau(r)}(\zeta))\eta=L_{\prec}^{*}(T_{r+\tau(r)}(\eta))\zeta.\end{align}
\end{enumerate}
for all $x\in A,~\zeta,\eta\in A^{*}$
 \end{pro}
 \begin{proof} It follows from Lemma \ref{In2}.
\end{proof}
By Eqs.~(\ref{Iv7})-(\ref{Iv8}), we get for all~$x\in A,~\zeta,\eta\in A^{*}$,
\begin{align}&\label{Iv9}L_{\circ}(x) T_{r+\tau(r)}(\zeta)=T_{r+\tau(r)}(L_{\succ}^{*}(x)\zeta),\ \ \
R_{\succ}(x) T_{r+\tau(r)}(\zeta)=T_{r+\tau(r)}(R_{\odot}^{*}(x)\zeta),
\\&\label{Iv10}R_{\circ}^{*} (T_{r+\tau(r)}(\zeta))\eta=-R_{\succ}^{*}(T_{r+\tau(r)}(\eta))\zeta, \ \ \
L_{\succ}^{*} (T_{r+\tau(r)}(\zeta))\eta=-L_{\odot}^{*}(T_{r+\tau(r)}(\eta))\zeta.
\end{align}

\begin{thm}\label{Ya1} Let $(A,\succ,\prec)$ be a Leibniz-dendriform
 algebra and $r=\sum_{i}a_i\otimes b_i\in A\otimes A$.
Assume that $r+\tau(r)$ is invariant. Then the following conditions are equivalent:
\begin{enumerate}
		\item $r$ is a solution of the LD-YBE $S(r)=0$.
\item $r$ is a solution of the equation $S_1(r)=0$.
\item $r$ is a solution of the equation $S_3(r)=0$.
\item $r$ is a solution of the equation $S_4(r)=0$.
\item $r$ is a solution of the equation $S(\tau(r))=r_{32}\circ r_{31}-r_{21}\odot r_{32}-r_{21}\succ r_{31}=0$,
\end{enumerate}
where the notations $S_i(r)~(i=1,2,3,4,5)$ appeared in Proposition \ref{Y4}.
\end{thm}

\begin{proof}
Note that
\begin{align*}&
S(\tau(r))=r_{32}\circ r_{31}-r_{21}\odot r_{32}-r_{21}\succ r_{31}
=-\sigma_{13}S_{3}(r),\\
&S_{1}(r)=r_{12}\odot r_{13}+r_{13}\circ r_{23}+r_{12}\succ r_{23}
\\=&\sigma_{12}S(r)+(r_{12}+r_{21})\odot r_{13}+(r_{12}+r_{21})\succ r_{23}
\\=&\sigma_{12}S(r)+\sum_{i}(R_{\odot}(a_i)\otimes I+I\otimes R_{\succ}(a_i))(r+\tau(r))\otimes b_i,
\\&
S_{4}(r) =-r_{13}\odot r_{12}+r_{12}\circ r_{23}+r_{13}\succ r_{23}
\\=&\sigma_{12}S_{3}(r)+(r_{12}+r_{21})\circ r_{23}-r_{13}\odot(r_{12}+r_{21})
\\=&\sigma_{12}S_{3}(r)+\sum_{i}(I\otimes R_{\circ}(a_i)-L_{\odot}(a_i)\otimes I)(r+\tau(r))\otimes b_i,
\\&
S_{3}(r)=r_{23}\odot r_{12}+r_{23}\succ r_{13}-r_{12}\circ r_{13}
\\=&-\sigma_{123}S_{1}(r)+r_{23}\odot(r_{12}+r_{21})+r_{23}\succ (r_{13}+r_{31})-(r_{12}+r_{21})\circ r_{13}
+r_{21}\circ (r_{13}+r_{31})
\\=&-\sigma_{123}S_{1}(r)+\sum_{i}(I\otimes L_{\odot}(a_i)-R_{\circ}(a_i)\otimes I)(r+\tau(r))\otimes b_i
\\&+(L_{\circ}(b_i)\otimes I\otimes I+I\otimes I\otimes L_{\succ}(b_i))(\tau\otimes I)[a_i\otimes(r+\tau(r))].\end{align*}
Combining Definition \ref{In1}, we get the conclusion.
\end{proof}

\begin{pro} \label{Da1}
Let $(A,\succ,\prec,\Delta_{\succ,r},\Delta_{\prec,r})$ be a coboundary Leibniz-dendriform bialgebra and $r\in A\otimes A$,
where the comultiplications
$\Delta_{\succ,r},\Delta_{\prec,r}$ are given by Eqs.~(\ref{CB1})-(\ref{CB2}).
 Then the Leibniz-dendriform algebra structure $(\succ_r,\prec_r)$ on $A^{*}$ is
given by
\begin{align}&\label{CE1}\zeta\succ_{r}\eta=L_{\circ}^{*}(T_{r}(\zeta))\eta-R_{\odot}^{*}(T_{\tau(r)}(\eta))\zeta,
\\&\label{CE2}\zeta\prec_{r}\eta=L_{\prec}^{*}(T_{\tau(r)}(\zeta))\eta-L_{\star}^{*}(T_{r}(\eta))\zeta.
\end{align}
where $\circ=\succ+\prec,~
 L_{\star}=L_{\circ}+R_{\circ},~R_{\odot}=R_{\succ}+L_{\prec}$.
\end{pro}
\begin{proof} For all $\zeta,\eta\in A^{*}$ and $x\in A$,
\begin{align*}\langle \zeta\succ_{r}\eta,x\rangle
&=\langle \zeta\otimes\eta,\Delta_{\succ,r}(x)\rangle
=\langle \zeta\otimes\eta,(L_{\odot}(x)\otimes I-I\otimes R_{\circ}(x))r
\rangle
\\&=\langle \zeta\otimes R_{\circ}^{*}(x)\eta,r
\rangle-\langle L_{\odot}^{*}(x)\zeta\otimes\eta, r\rangle
\\&=\langle T_{r}(\zeta),R_{\circ}^{*}(x)\eta
\rangle-\langle T_{\tau(r)}(\eta),L_{\odot}^{*}(x)\zeta\rangle
\\&=-\langle T_{r}(\zeta)\circ x,\eta
\rangle+\langle x\odot T_{\tau(r)}(\eta),\zeta\rangle
\\&=\langle x,L_{\circ}^{*}(T_{r}(\zeta))\eta-R_{\odot}^{*}(T_{\tau(r)}(\eta))\zeta\rangle.
\end{align*}
It yields that Eq.~(\ref{CE1}) holds. Likewise, Eq.~(\ref{CE2}) follows.
\end{proof}

\begin{thm} Let $(A,\succ,\prec)$ be a Leibniz-dendriform algebra and $r=\sum_{i}a_i\otimes b_i\in A\otimes A$.
 Assume that
$\Delta_{\succ,r},\Delta_{\prec,r}$ are given by Eqs.~(\ref{CB1})-(\ref{CB2}). If $r$ is a solution of the
LD-YBE in $(A,\succ,\prec)$ and $r+\tau(r)$ is invariant.
Then $(A,\succ,\prec,\Delta_{\succ,r},\Delta_{\prec,r})$ is a Leibniz-dendriform bialgebra.
\end{thm}
\begin{proof}
Since $r$ is a solution of the
LD-YBE in $(A,\succ,\prec)$ and $r+\tau(r)$ is invariant,
Eqs.~(\ref{Iv7})-(\ref{Iv9}) hold and
\begin{equation*}
S(r)=S_{1}(r)=S_{4}(r)=S_{2}(r)=0.
\end{equation*}
Using Eqs.~(\ref{R10}), (\ref{Iv7}) and (\ref{Iv9}), we have for all $\zeta,\eta\in A^{*}$
\begin{align*}&
\langle \zeta\otimes \eta,
( L_{\odot}(x)R_{\succ}(a_i)\otimes I+R_{\odot}(a_i)\otimes L_{\odot}(x)+L_{\odot}(x\odot a_i)\otimes I-I\otimes R_{\prec}(x\odot a_i))( r+ \tau (r))\rangle
\\=&-\langle R_{\succ}^{*}(a_i)L_{\odot}^{*}(x)\zeta\otimes \eta, r+ \tau (r)\rangle
+\langle R_{\odot}^{*}(a_i)\zeta\otimes L_{\odot}^{*}(x)\eta, r+ \tau (r)\rangle
\\&-\langle L_{\odot}^{*}(x\odot a_i)\zeta\otimes \eta, r+ \tau (r)\rangle
+\langle \zeta\otimes R_{\prec}^{*}(x\odot a_i)\eta, r+ \tau (r)\rangle
\\=&-\langle T_{r+ \tau (r)}(R_{\succ}^{*}(a_i)L_{\odot}^{*}(x)\zeta), \eta\rangle
+\langle T_{r+ \tau (r)}(R_{\odot}^{*}(a_i)\zeta), L_{\odot}^{*}(x)\eta\rangle
\\&-\langle T_{r+ \tau (r)}(L_{\odot}^{*}(x\odot a_i)\zeta), \eta\rangle
+\langle T_{r+ \tau (r)}(\zeta), R_{\prec}^{*}(x\odot a_i)\eta\rangle
\\=&\langle R_{\odot}(a_i)R_{\circ}(x)T_{r+ \tau (r)}(\zeta), \eta\rangle
-\langle L_{\odot}(x)R_{\succ}(a_i)T_{r+ \tau (r)}(\zeta), \eta\rangle
\\&+\langle R_{\circ}(x\odot a_i)T_{r+ \tau (r)}(\zeta), \eta\rangle
-\langle R_{\prec}(x\odot a_i)T_{r+ \tau (r)}(\zeta), \eta\rangle
\\=&0,
\end{align*}
which implies that
\begin{equation} \label{Iy}
(L_{\odot}(x)R_{\succ}(a_i)\otimes I+R_{\odot}(a_i)\otimes L_{\odot}(x)+L_{\odot}(x\odot a_i)\otimes I-I\otimes R_{\prec}(x\odot a_i))( r+ \tau (r))=0.
\end{equation}
Note that
\begin{align*}&
r_{23}\circ r_{13}-r_{12}\odot r_{23}+r_{21}\succ r_{13}
= S(r)+(r_{21}+r_{12})\succ r_{13}
\\=& S(r)+\sum_{i}(R_{\succ}(a_i)\otimes I)(r+\tau(r)) \otimes b_i,
\\&
r_{21}\odot r_{13}-r_{13}\circ r_{23}-r_{12}\succ r_{23}
=-S_{1}(r)+(r_{21}+r_{12})\odot r_{13}
\\=&-S_{1}(r)+\sum_{i}(R_{\odot}(a_i)\otimes I)(r+\tau(r)) \otimes b_i,
\\&
r_{12}\circ r_{23}-r_{13}\odot r_{12}-r_{13}\star r_{21}+r_{21}\prec r_{23}
+r_{13}\circ r_{23}
\\=&S_{4}(r)+S_{2}(r)+(r_{21}+r_{12})\prec r_{23}-(r_{21}+r_{12})\star r_{13}
\\=&S_{4}(r)+S_{2}(r)+\sum_{i}(I\otimes R_{\prec}(a_i)-L_{\star}(a_i)\otimes I)(r+\tau(r)) \otimes b_i.
\end{align*}
Therefore, using Eq.~(\ref{Iy}), we have
\begin{align*}&(L_{\odot}(x)\otimes I \otimes I)(r_{23}\circ r_{13}-r_{12}\odot r_{23}+r_{21}\succ r_{13})
+(I\otimes L_{\odot}(x)\otimes I)(r_{21}\odot r_{13}-r_{13}\circ r_{23}-r_{12}\succ r_{23})
\\&+(I\otimes I\otimes R_{\circ}(x))(r_{12}\circ r_{23}-r_{13}\odot r_{12}-r_{13}\star r_{21}+r_{21}\prec r_{23}
+r_{13}\circ r_{23})
\\&+\sum_{i}(L_{\odot}(x\odot a_i)\otimes I-I\otimes R_{\prec}(x\odot a_i))(r+\tau(r)) \otimes b_i\\=&
(L_{\odot}(x)\otimes I \otimes I)S(r)+\sum_{i}(L_{\odot}(x)R_{\succ}(a_i)\otimes I)(r+\tau(r)) \otimes b_i
-(I\otimes L_{\odot}(x)\otimes I)S_{1}(r)\\&+\sum_{i}(R_{\odot}(a_i)\otimes L_{\odot}(x))(r+\tau(r)) \otimes b_i
+(I\otimes I\otimes R_{\circ}(x))[S_{4}(r)+S_{2}(r)+\sum_{i}(I\otimes R_{\prec}(a_i)\\&-L_{\star}(a_i)\otimes I)(r+\tau(r)) \otimes b_i]
+\sum_{i}(L_{\odot}(x\odot a_i)\otimes I-I\otimes R_{\prec}(x\odot a_i))(r+\tau(r)) \otimes b_i\\=&0,
\end{align*}
that is, Eq.~(\ref{CB3}) holds. Analogously, Eqs.~(\ref{CB4})-(\ref{CB8}) hold.
The proof is finished.
\end{proof}

\begin{defi} \label{Qt1}
 Let $(A,\succ,\prec)$ be a Leibniz-dendriform algebra and $r\in A\otimes A$. If $r$ is a solution of the LD-YBE in $(A,\succ,\prec)$
 and $r+\tau(r)$ is invariant, then the Leibniz-dendriform
  bialgebra $(A, \succ,\prec,\Delta_{\succ,r},\Delta_{\prec,r})$ induced by $r$ is called a {\bf quasi-triangular} Leibniz-dendriform
 bialgebra. In particular, if $r$ is skew-symmetric, $(A, \succ,\prec,\Delta_{\succ,r},\Delta_{\prec,r})$
is called a {\bf triangular} Leibniz-dendriform bialgebra, where $\Delta_{\succ,r}$ and $\Delta_{\prec,r}$
are given by Eqs.~(\ref{CB1})-(\ref{CB2}) respectively.
\end{defi}

\begin{thm}\label{QB0}  Let $(A,\succ,\prec)$ be a Leibniz-dendriform algebra and
$r\in A\otimes A$. Suppose that $r+\tau(r)$ is invariant. Then the following conditions are equivalent:
 \begin{enumerate}
\item $r$ is a solution of the LD-YBE in $(A,\succ,\prec)$.
 \item $(A^{*},\circ_r)$ is a Leibniz algebra and the linear maps
$T_{r},-T_{\tau(r)}$ are both Leibniz algebra
 homomorphisms from $(A^{*},\circ_r)$ to $(A,\circ)$.
 \item $(A^{*},\succ_r,\prec_r)$ is a Leibniz-dendriform algebra and the linear maps
$T_{r},-T_{\tau(r)}$ are both Leibniz-dendriform algebra
 homomorphisms from $(A^{*},\succ_r,\prec_r)$ to $(A,\succ,\prec)$.
 \end{enumerate}
\end{thm}
\begin{proof}
Since $r + \tau(r)$ is invariant, Eqs.~(\ref{Iv8})-(\ref{Iv10}) hold.
Then, from Eqs.~(\ref{CE1})-(\ref{CE2}), (\ref{Iv8}) and (\ref{Iv10}), we derive
\begin{align*}&\zeta\circ_{r}\eta=\zeta\succ_{r}\eta+\zeta\prec_{r}\eta
\\=&L_{\circ}^{*}(T_{r}(\zeta))\eta-R_{\odot}^{*}(T_{\tau(r)}(\eta))\zeta
+L_{\prec}^{*}(T_{\tau(r)}(\zeta))\eta-L_{\star}^{*}(T_{r}(\eta))\zeta
\\=&L_{\prec}^{*}(T_{r+\tau(r)}(\zeta))\eta+L_{\succ}^{*}(T_{r}(\zeta))\eta
-R_{\odot}^{*}(T_{r+\tau(r)}(\eta))\zeta-L_{\odot}^{*}(T_{r}(\eta))\zeta
\\=&L_{\star}^{*}(T_{r+\tau(r)}(\eta))\zeta+L_{\succ}^{*}(T_{r}(\zeta))\eta
-R_{\odot}^{*}(T_{r+\tau(r)}(\eta))\zeta-L_{\odot}^{*}(T_{r}(\eta))\zeta
\\=&L_{\succ}^{*}(T_{r}(\zeta))\eta+L_{\odot}^{*}(T_{\tau(r)}(\eta))\zeta
\\=&L_{\succ}^{*}(T_{r}(\zeta))\eta+L_{\odot}^{*}(T_{\tau(r)}(\eta))\zeta
-L_{\succ}^{*}(T_{r+\tau(r)}(\zeta))\eta-L_{\odot}^{*}(T_{r+\tau(r)}(\eta))\zeta
\\=&-L_{\succ}^{*}(T_{\tau(r)}(\zeta))\eta-L_{\odot}^{*}(T_{r}(\eta))\zeta,\end{align*}
which yields the identities:
\begin{align*}
\zeta \circ_r \eta = L_{\succ}^{*}(T_r(\zeta))\eta + L_{\odot}^{*}(T_{\tau(r)}(\eta))\zeta
= -L_{\succ}^{*}(T_{\tau(r)}(\zeta))\eta - L_{\odot}^{*}(T_r(\eta))\zeta.
\end{align*}
It now follows from Item (b) and Item (e) of Proposition~\ref{Y4}, Theorem~\ref{Ya1} and Proposition~\ref{Da1}
that Item (a) $\Longleftrightarrow$ Item (b). Observing that $S_{5}(r)=S(r)-S_{3}(r)$.
A similar argument shows that Item (a) $\Longleftrightarrow$ Item (c).
The proof is completed.
\end{proof}

\begin{cor} \label{QB} Let $(A,\succ,\prec,\Delta_{\succ,r},\Delta_{\prec,r})$ be a quasi-triangular Leibniz-dendriform bialgebra.
 Then \begin{enumerate}
\item $T_{r},-T_{\tau(r)}$ are both Leibniz algebra
 homomorphisms from $(A^{*},\circ_r)$ to $(A,\circ)$.
 \item $T_{r},-T_{\tau(r)}$ are Leibniz-dendriform algebra
 homomorphisms from $(A^{*},\succ_r,\prec_r)$ to $(A,\succ,\prec)$.
 \end{enumerate}
\end{cor}

\begin{proof}
It follows by Theorem \ref{QB0} and Definition \ref{Qt1}.
\end{proof}

\subsection{Factorizable Leibniz-dendriform bialgebras}

This section introduces factorizable Leibniz-dendriform bialgebras,
defined as a particular subclass of quasi-triangular Leibniz-dendriform bialgebras.
We then show that the double of any Leibniz-dendriform bialgebra is itself a
 factorizable Leibniz-dendriform bialgebra.

\begin{defi} \label{Qt} A quasi-triangular Leibniz-dendriform
 bialgebra $(A, \succ,\prec, \Delta_{\succ,r}, \Delta_{\prec,r})$ is
called factorizable if
 the linear map $T_{r+\tau(r)}:A^{*}\longrightarrow A$ given by Eq.~(\ref{YE2}) is a linear isomorphism of vector spaces,
where $\Delta_{\succ,r}, \Delta_{\prec,r} $
are defined by Eqs. (\ref{CB1})-(\ref{CB2}).\end{defi}

In view of Item (e) of Theorem \ref{Ya1} and Definition \ref{Qt}, we have the following statement:

\begin{pro} \label{Qf} Assume that $(A, \succ,\prec, \Delta_{\succ,r}, \Delta_{\prec,r})$ is a quasi-triangular
 (factorizable) Leibniz-dendriform bialgebra. Then $(A, \succ,\prec, \Delta_{\succ,\tau(r)}, \Delta_{\prec,\tau(r)})$
 is also a quasi-triangular
 (factorizable) Leibniz-dendriform bialgebra.\end{pro}

\begin{pro}
Let $(A, \succ,\prec, \Delta_{\succ,r}, \Delta_{\prec,r})$ be a factorizable Leibniz-dendriform bialgebra.
Then $\mathrm{Im}(T_{r}\oplus T_{\tau(r)})$ is a Leibniz-dendriform subalgebra of the direct sum Leibniz-dendriform
algebra $A\oplus A$,
which is isomorphic to the Leibniz-dendriform algebra $(A^{*},\succ_{r},\prec_{r})$. Furthermore, any $x\in A$ has a unique
decomposition $x=x_1-x_2$, where $(x_1,x_2)\in \mathrm{Im}(T_{r}\oplus T_{\tau(r)})$ and
\begin{equation*}T_{r}\oplus T_{\tau(r)}:A^{*}\longrightarrow A\oplus A,~~(T_{r}\oplus T_{\tau(r)})(\zeta)=(T_{r}(\zeta),
 -T_{\tau(r)}(\zeta)),~\forall~\zeta\in A^{*}. \end{equation*}
\end{pro}

\begin{proof} In view of Definition \ref{Qt} and Corollary \ref{QB}, $T_{r}, -T_{\tau(r)}$ are Leibniz-dendriform algebra
homomorphisms and $T_{r+\tau(r)}$ is a linear isomorphism. It follows that
$T_{r}\oplus T_{\tau(r)}$ is a Leibniz-dendriform algebra homomorphism and $\mathrm{Ker}(T_{r}\oplus T_{\tau(r)})=0$.
Therefore, $\mathrm{Im}(T_{r}\oplus T_{\tau(r)})$ is isomorphic to $(A^{*},\succ_{r},\prec_{r})$ as Leibniz-dendriform algebras.
For all $x\in A$,
\begin{equation*}x=T_{r+\tau(r)}T_{r+\tau(r)}^{-1}(x)=T_{r}T_{r+\tau(r)}^{-1}(x)+T_{\tau(r)}T_{r+\tau(r)}^{-1}(x)=x_1-x_2
,\end{equation*}
where $x_1=T_{r}T_{r+\tau(r)}^{-1}(x)\in \mathrm{Im}(T_{s}),
x_2=-T_{\tau(r)}T_{r+\tau(r)}^{-1}(x)\in \mathrm{Im}(-T_{\tau(r)})$.
The proof is finished.
\end{proof}
\begin{pro}
 Let $(A, \succ,\prec, \Delta_{\succ,r}, \Delta_{\prec,r})$
 be a factorizable Leibniz-dendriform bialgebra. Then the double Leibniz-dendriform algebra $(D=A\oplus A^{*}, \succ_D,\prec_D)$
is isomorphic to the direct sum $A\oplus A$ of Leibniz-dendriform algebras.
\end{pro}

\begin{proof} Since $(A, \succ,\prec, \Delta_{\succ,r}, \Delta_{\prec,r})$
 is a factorizable Leibniz-dendriform bialgebra,
$T_{r+\tau(r)}$ is a linear isomorphism and $r+\tau(r)$ is invariant.
Define $\varphi:A\oplus A^{*}\longrightarrow A\oplus A$ by
\begin{equation}\label{FD1}\varphi(x,\zeta)=(x+T_{r}(\zeta),x-T_{\tau(r)}(\zeta)),~\forall~(x,\zeta)\in A\oplus A^{*}.\end{equation}
Then $\varphi$ is a bijection.
Using Eqs.~(\ref{CE1})-(\ref{CE2}) and (\ref{Iv10}), we have
\begin{align*}&\langle\eta,(R_{\succ_{A^*}}^{*}+L_{\prec_{A^*}}^{*})(\zeta)x+T_{r}(L_{\circ}^{*}(x)\zeta)\rangle
= -\langle \eta\succ_{r} \zeta+\zeta\prec_{r} \eta,x\rangle-\langle x\circ T_{\tau(r)}(\eta),\zeta\rangle
\\=& -\langle L_{\circ}^{*}(T_{r}(\eta))\zeta-R_{\odot}^{*}(T_{\tau(r)}(\zeta))\eta+
L_{\prec}^{*}(T_{\tau(r)}(\zeta))\eta-L_{\star}^{*}(T_{r}(\eta))\zeta,x\rangle
+\langle R_{\circ}^{*}(T_{\tau(r)}(\eta))\zeta,x\rangle
\\=&\langle R_{\circ}^{*}(T_{r+\tau(r)}(\eta))\zeta+ R_{\succ}^{*}(T_{\tau(r)}(\zeta))\eta,x\rangle
\\=&-\langle  R_{\succ}^{*}(T_{r}(\zeta))\eta,x\rangle
\\=&\langle  \eta,x\succ T_{r}(\zeta)\rangle.
\end{align*}
It follows that
\begin{equation}\label{FD2} (R_{\succ_{A^*}}^{*}+L_{\prec_{A^*}}^{*})(\zeta)x+T_{r}(L_{\circ}^{*}(x)\zeta)
=x\succ T_{r}(\zeta).\end{equation}
By the same token,
\begin{equation}\label{FD3}(R_{\succ_{A^*}}^{*}+L_{\prec_{A^*}}^{*})(\zeta)x-T_{\tau(r)}(L_{\circ}^{*}(x)\zeta)
=-x\succ T_{\tau(r)}(\zeta).\end{equation}
Combining Eqs.~(\ref{Db1}) and (\ref{FD1})-(\ref{FD3}),
\begin{align*}&\varphi(x\succ_D \zeta)=\varphi((R_{\succ_{A^*}}^{*}+L_{\prec_{A^*}}^{*})(\zeta)x+L_{\circ}^{*}(x)\zeta)
\\
=&((R_{\succ_{A^*}}^{*}+L_{\prec_{A^*}}^{*})(\zeta)x+T_{r}(L_{\circ}^{*}(x)\zeta),
(R_{\succ_{A^*}}^{*}+L_{\prec_{A^*}}^{*})(\zeta)x-T_{\tau(r)}(L_{\circ}^{*}(x)\zeta))
\\=&(x\succ T_{r}(\zeta),-x\succ T_{\tau(r)}(\zeta))
=(x,x)\succ (T_{r}(\zeta),-T_{\tau(r)}(\zeta))
\\=&\varphi(x)\succ \varphi(\zeta).
\end{align*}
Similarly, we have $\varphi(\zeta\succ_D x)=\varphi(\zeta)\succ \varphi(x)$.
Hence, $\varphi((x,\zeta)\succ_D (y,\eta))=\varphi(x,\zeta)\succ \varphi(y,\eta)$ for
all $(x,\zeta), (y,\eta)\in A\oplus A^{*}$.
An analogous proof shows that $\varphi((x,\zeta)\prec_D (y,\eta))=\varphi(x,\zeta)\prec \varphi(y,\eta)$.
Therefore, $\varphi$ is an isomorphism of Leibniz-dendriform algebras.
The proof is completed.
\end{proof}

\begin{thm} \label{Ft}
Let $(A, \succ,\prec, \Delta_{\succ}, \Delta_{\prec})$ be a Leibniz-dendriform
 bialgebra. Assume that $\{e_1,...,e_n\}$ is a basis of $A$ and $\{e^{*}_1,...,e^{*}_n\}$ is the dual basis.
 Let \begin{equation*}r=\sum_{i=1}^{n}e_{i}\otimes e_{i}^{*}\in A\otimes A^{*}\subseteq D\otimes D.\end{equation*}
Then $(D,\succ_D,\prec_D, \Delta_{\succ,r},\Delta_{\prec,r})$ with
$\Delta_{\succ,r},\Delta_{\prec,r}$ given by Eqs.~(\ref{CB1})-(\ref{CB2})
 is a factorizable Leibniz-dendriform bialgebra.
\end{thm}
\begin{proof}
 We begin by proving that the element $r + \tau(r) = \sum_{i=1}^{n} (e_i \otimes e_i^* + e_i^* \otimes e_i)$ is invariant. Indeed,
by Eqs.~(\ref{Db1}) and (\ref{Db2}), for all $x\in A$, we have
\begin{align*}&(L_{\odot_D}(x)\otimes I-I\otimes R_{\circ_D}(x))\sum_{i}^{n}(e_{i}\otimes e_{i}^{*}+e_{i}^{*}\otimes e_{i})
\\=&\sum_{i}^{n}x\odot_{D}e_{i}\otimes e_{i}^{*}-
e_{i}\otimes e_{i}^{*}\circ_D x+x\odot_{D}e_{i}^{*}\otimes e_{i}
-e_{i}^{*}\otimes e_{i}\circ_{D}x
\\=&\sum_{i,j=1}^{n}x\odot_{D}e_{i}\otimes e_{i}^{*}+
e_{i}\otimes L_{\odot}^{*}(x)e_{i}^{*}-e_{i}\otimes L_{\succ_{A^*}}^{*}(e_{i}^{*})x
+R_{\succ_{A^*}}^{*}(e_{i}^{*})x\otimes e_{i}-R_{\circ}^{*}(x)e_{i}^{*}\otimes e_{i}
-e_{i}^{*}\otimes e_{i}\circ x.
\end{align*}
Note that
\begin{align*}& \sum_{i}^{n}x\odot_{D}e_{i}\otimes e_{i}^{*}+
e_{i}\otimes L_{\odot}^{*}(x)e_{i}^{*}=0,
\\&\sum_{i}^{n}-e_{i}\otimes L_{\succ_{A^*}}^{*}(e_{i}^{*})x
+R_{\succ_{A^*}}^{*}(e_{i}^{*})x\otimes e_{i}=0,
\\&\sum_{i}^{n}R_{\circ}^{*}(x)e_{i}^{*}\otimes e_{i}
+e_{i}^{*}\otimes e_{i}\circ x=0.
\end{align*}
Thus, $(L_{\odot_D}(x)\otimes I-I\otimes R_{\circ_D}(x))(r+\tau(r))=0.$
Similarly, $(L_{\star}(x)\otimes I-I\otimes R_{\prec}(x))(r+\tau(r))=0.$
By duality, we get
\begin{equation*}(L_{\odot_D}(\zeta)\otimes I-I\otimes R_{\circ_D}(\zeta))(r+\tau(r))=0,
\ \ \
(L_{\star}(\zeta)\otimes I-I\otimes R_{\prec}(\zeta))(r+\tau(r))=0\end{equation*}
for all $\zeta\in A^{*}$. This establishes the invariance of $r + \tau(r)$.
We now prove that $r$ satisfies the LD-YBE in $(D,\succ_D,\prec_D)$.
According to Eqs.~(\ref{Db1}) and (\ref{Db2}), we get
\begin{align*}& r_{23}\circ_D r_{13}-r_{12}\odot_D r_{23}-r_{12}\succ_D r_{13}
\\=&\sum_{i,j=1}^{n}e_{j}\otimes e_{i} \otimes e_{i}^{*}\circ_D e_{j}^{*}
-e_{i}\otimes e_{i}^{*}\odot_D e_j\otimes e_{j}^{*}
-e_{i}\succ_D e_{j}\otimes e_{i}^{*}\otimes e_{j}^{*}
\\=&\sum_{i,j=1}^{n}e_{j}\otimes e_{i} \otimes e_{i}^{*}\circ_D e_{j}^{*}
+e_{i}\otimes R_{\circ_{A^*}}^{*}(e_{i}^{*}) e_j\otimes e_{j}^{*}
-e_{i}\otimes R_{\succ}^{*}(e_j)e_{i}^{*}\otimes e_{j}^{*}
-e_{i}\succ_{D} e_{j}\otimes e_{i}^{*}\otimes e_{j}^{*}
=0.\end{align*}
Consequently, $(D, \succ_D, \prec_D, \Delta_{\succ,r}, \Delta_{\prec,r})$ is a quasi-triangular Leibniz-dendriform bialgebra.
Finally, consider the linear maps $T_{r},T_{\tau(r)}:D^{*}\longrightarrow D$ are defined respectively by
$T_{r}(\zeta,x)=\zeta,~T_{\tau(r)}(\zeta,x)=-x$ for all $x\in A,\zeta\in A^{*}$.
Since the map $T_{r+\tau(r)}(\zeta, x) = (\zeta, -x)$ is a linear isomorphism, the bialgebra is in fact factorizable.
\end{proof}

\section{Relative Rota-Baxter operators and quadratic Rota-Baxter Leibniz-dendriform algebras }

\subsection{Relative Rota-Baxter operators and the Leibniz-dendriform Yang-Baxter equation}
This section explores the connection between solutions of the LD-YBE that possess invariant symmetric
parts and relative Rota-Baxter operators of weight defined on Leibniz-dendriform algebras.

\begin{defi} Let $(A,\succ,\prec)$ be a Leibniz-dendriform algebra and
$(V,\succ_V,\prec_V,l_{\succ},r_{\succ},l_{\prec},r_{\prec})$ be an A-Leibniz-dendriform algebra.
A relative Rota-Baxter operator $T$ of weight $\lambda$ on $(A,\succ,\prec)$ associated to
 $(V,\succ_V,\prec_V,l_{\succ},r_{\succ},l_{\prec},r_{\prec})$
   is a linear map $T:V\longrightarrow A$ satisfying
\begin{align*}&T(u)\succ T(v)=T (l_{\succ}(T(u))v+r_{\succ}(T(v))u+\lambda u\succ_V v),\\&
T(u)\prec T(v)=T (l_{\prec}(T(u))v+r_{\prec}(T(v))u+\lambda u\prec_V v),~~\forall~u,v\in V.\end{align*}
When $u\succ_Vv=u\prec_Vv=0$ for all $u,v\in V$, then $T$ is simply a relative Rota-Baxter operator ($\mathcal O$-operator) on
$(A,\succ,\prec)$ associated to a representation $(V,l_{\succ},r_{\succ},l_{\prec},r_{\prec})$.
\end{defi}

\begin{pro}
Let $(A,\succ,\prec)$ be a Leibniz-dendriform algebra and $r\in A\otimes A$ be symmetric and invariant. Define the
multiplications $\succeq_r,\preceq_r:A^{*}\otimes A^{*}\longrightarrow A^{*}$ by
\begin{align}&\label{AD1} \zeta\succeq_r\eta=L_{\circ}^{*}(T_{r}(\zeta))\eta=R_{\odot}^{*}(T_{r}(\eta))\zeta,
 \\&\label{AD2}\zeta\preceq_r\eta=-L_{\prec}^{*}(T_{r}(\zeta))\eta=-L_{\star}^{*}(T_{r}(\eta))\zeta
,~~\forall~\zeta,\eta\in A^{*}.
\end{align}
Then $(A^{*},\succeq_r,\preceq_r,L_{\circ}^{*},R_{\odot}^{*},-L_{\prec}^{*}, -L_{\star}^{*})$ is an A-Leibniz-dendriform algebra
and $(A^{*},\circ_r,L_{\succ}^*,-L_{\odot}^*)$ is an A-Lebniz algebra, where
\begin{equation}\label{AD3} \zeta\circ_r\eta=-L_{\odot}^{*}(T_r(\zeta))\eta=R_{\prec}^{*}(T_r(\eta))\zeta,~~\forall~\zeta,\eta\in A^{*}.
\end{equation}
\end{pro}

\begin{proof} Using Eqs.~(\ref{Lr1}), (\ref{Iv7}) and (\ref{AD1})-(\ref{AD2}),
we find for all $\zeta, \eta, \theta \in A^{}$ that
\begin{align*} &\zeta\succeq_r(\eta\succeq_r \theta)-\eta\succeq_r(\zeta\succeq_r \theta)-(\zeta\circ_r\eta)\succeq_r \theta
 \\=&L_{\circ}^{*}(T_{r}(\zeta))L_{\circ}^{*}(T_{r}(\eta))\theta-L_{\circ}^{*}(T_{r}(\eta))L_{\circ}^{*}(T_{r}(\zeta))
 -L_{\circ}^{*}(T_{r}(R_{\prec}^{*}(T_{r}(\eta))\zeta))
 \\=&L_{\circ}^{*}(T_{r}(\zeta))L_{\circ}^{*}(T_{r}(\eta))\theta-L_{\circ}^{*}(T_{r}(\eta))L_{\circ}^{*}(T_{r}(\zeta))\theta
 +L_{\circ}^{*}(T_{r}(\zeta)\star T_{r}(\eta))\theta
\\=&0,\end{align*}
which implies that Eq.~(\ref{Ld1}) holds for $\succeq_r$ and $\circ_r$.
The verification of Eqs.~(\ref{Ld2})-(\ref{Ld3}) and the identities in Remark \ref{La}
for $\succeq_r$, $\preceq_r$ and $\circ_r$ is similar, relying on
 Eqs.~(\ref{R1})-(\ref{R9}), (\ref{Iv6})-(\ref{Iv10}) and (\ref{AD1})-(\ref{AD3}). Thus, the proof is completed.

\end{proof}

The following result ia a generalization of Theorem \ref{YE3}.
\begin{thm} \label{Al}
Let $(A,\succ,\prec)$ be a Leibniz-dendriform algebra and $r\in A\otimes A$. Assume that $r+\tau(r)$ is invariant.
Then the following
conditions are equivalent.
\begin{enumerate}
 \item $r$ is a solution of the LD-YBE in $(A,\succ,\prec)$
 such that $(A,\succ,\prec,\Delta_{\succ,r},\Delta_{\prec,r})$
with $\Delta_{\succ,r},\Delta_{\prec,r}$ given by Eqs.~(\ref{CB1})-(\ref{CB2}) is a quasi-triangular Leibniz-dendriform bialgebra.
\item $T_r$ is a relative Rota-Baxter operator of weight $-1$ on $(A,\succ,\prec)$
 with respect to the A-Leibniz-dendriform algebra
  $(A^{*},\succeq_{r+\tau(r)},\preceq_{r+\tau(r)},L_{\circ}^{*},R_{\odot}^{*},-L_{\prec}^{*}, -L_{\star}^{*})$,
 that is,
 \begin{align}\label{AD5}&T_{r}(\zeta)\prec T_{r}(\eta)=T_{r}(-L_{\prec}^{*}(T_{r}(\zeta))\eta
-L_{\star}^{*}(T_{r}(\eta))\zeta-\zeta\preceq_{r+\tau(r)}\eta), \\&
 \label{AD6}T_{r}(\zeta)\succ T_{r}(\eta)=T_{r}(L_{\circ}^*(T_{r}(\zeta))\eta+R_{\odot}^*(T_{r}(\eta))\zeta
-\zeta\succeq_{r+\tau(r)}\eta).
 \end{align}
\item $T_r$ is a relative Rota–Baxter operator of weight $-1$ on $(A,\circ)$
 with respect to the A-Leibniz algebra $(A^{*},\circ_{r+\tau(r)},L_{\succ}^*,-L_{\odot}^*)$, that is,
 \begin{align}\label{AD8} &
 T_{r}(\zeta)\circ T_{r}(\eta)=T_{r}(L_{\succ}^*(T_{r}(\zeta))\eta-L_{\odot}^*(T_{r}(\eta))\zeta-\zeta\circ_{r+\tau(r)}\eta),
 \end{align}
  \end{enumerate}
for all $\zeta,\eta\in A^{*}$, where $\succeq_{r+\tau(r)},~\preceq_{r+\tau(r)}$ and $\circ_{r+\tau(r)}$ are given respectively by Eqs.~(\ref{AD1})-(\ref{AD3})
\end{thm}

 \begin{proof}
 In view of Item (b) of Proposition \ref{Y4} and Theorem \ref{Ya1}, if $r+\tau(r)$ is invariant, then
 $r$ is a solution of the LD-YBE in $(A,\succ,\prec)$ if and only if
 \begin{equation*}T_{r}(\zeta)\circ T_{r}(\eta)=T_{r}(-L_{\succ}^*(T_{\tau(r)}(\zeta))\eta-L_{\odot}^*(T_{r}(\eta))\zeta).\end{equation*}
 Observe that
\begin{align*}T_{r}(\zeta)\circ T_{r}(\eta)&=T_{r}(-L_{\succ}^*(T_{\tau(r)}(\zeta))\eta-L_{\odot}^*(T_{r}(\eta))\zeta)
\\&=T_{r}(L_{\succ}^*(T_{r}(\zeta))\eta-L_{\odot}^*(T_{r}(\eta))\zeta-L_{\succ}^*(T_{r+\tau(r)}(\zeta))\eta)
\\&=T_{r}(L_{\succ}^*(T_{r}(\zeta))\eta-L_{\odot}^*(T_{r}(\eta))\zeta-\zeta\circ_{r+\tau(r)}\eta).\end{align*}
Thus, Item (a) $\Longleftrightarrow$ Item (c).
Similarly, combining Item (c) and Item (d) of Proposition \ref{Y4} and Theorem \ref{Ya1}, we get that
Item (a) $\Longleftrightarrow$ Item (b).
  \end{proof}
 If $r\in A\otimes A$ is skew-symmetric, then $r+\tau(r)=0$. Thus, Theorem \ref{YE3} is obtained.

\subsection{Quadratic Rota-Baxter Leibniz-dendriform algebras and factorizable Leibniz-dendriform bialgebras}

In this section, we first introduce a notion of quadratic Rota-Baxter Leibniz-dendriform algebras.
Then we characterize the relationship between factorizable Leibniz-dendriform
bialgebras and quadratic Rota-Baxter Leibniz-dendriform algebras.

\begin{defi} Let $(A,\succ,\prec,P)$ be a Rota-Baxter Leibniz-dendriform algebra of weight $\lambda$
and $(A,\succ,\prec,\omega)$ a quadratic Leibniz-dendriform algebra. Then $(A,\succ,\prec,P,\omega)$
is called a \textbf{quadratic Rota-Baxter Leibniz-dendriform algebra of weight $\lambda$} if the following condition holds:
\begin{equation} \label{Fs}\omega (P(x),y)+\omega(x, P(y))+\lambda\omega(x,y)=0, ~\forall~x, y \in A.\end{equation}
\end{defi}

\begin{defi} Let $(A,\circ,P)$ be a Rota-Baxter Leibniz algebra of weight $\lambda$ and
 $(A,\circ,\omega )$ be a symplectic Leibniz algebra. In addition, if the following condition holds:
 \begin{equation} \label{Fs1}\omega (P(x),y)+\omega(x, P(y))+\lambda\omega(x,y)=0, ~\forall~x, y \in A.\end{equation}
Then $(A,\circ ,P,\omega)$ is called a \textbf{ Rota-Baxter symplectic Leibniz algebra of
weight $\lambda$}.
\end{defi}

Since  \begin{align*} &\lambda\omega(x,y)+\omega (-\lambda(x)- P(x),y)+\omega(x, -\lambda(y)- P(y))\\=&
-\lambda\omega(x,y)-\omega (P(x),y)-\omega(x,  P(y)), ~\forall~x, y \in A,\end{align*}
the following conclusions are clear.

\begin{pro} \label{Fb2} Let $(A,\succ,\prec,\omega)$ be a quadratic Leibniz-dendriform algebra and let $P:A\longrightarrow A$ be a linear
map. Then $(A,\succ,\prec,P,\omega)$ is a quadratic Rota-Baxter Leibniz-dendriform algebra of weight $\lambda$ if and only if
$(A,\succ,\prec,-\lambda I-P,\omega)$ is a quadratic Rota-Baxter Leibniz-dendriform algebra of weight $\lambda$.
\end{pro}

\begin{pro} Let $(A,\circ,\omega)$ be a symplectic Leibniz algebra and let $P:A\longrightarrow A$ be a linear
map. Then $(A,\circ ,P,\omega)$ is a Rota-Baxter symplectic Leibniz algebra of weight $\lambda$ if and only if
$(A,\circ ,-\lambda I-P,\omega)$ is a Rota-Baxter symplectic Leibniz algebra of weight $\lambda$.
\end{pro}

\begin{thm}\label{Fb0} Suppose that $(A,\circ,P,\omega)$ is a Rota-Baxter symplectic Leibniz algebra of
weight $\lambda$. Then $(A,\succ,\prec,P,\omega)$ is a quadratic Rota-Baxter Leibniz-dendriform
algebra of weight $\lambda$, where $\succ,\prec$ are defined by Eq.~(\ref{C1}).
 On the other hand, let $(A,\succ,\prec,P,\omega)$ be a quadratic Rota-Baxter Leibniz-dendriform algebra of weight $\lambda$.
Then $(A,\circ,P,\omega)$ is a Rota-Baxter symplectic Leibniz algebra of weight $\lambda$, where $\circ=\succ+\prec$.
\end{thm}

\begin{proof} Let $(A,\circ,P,\omega)$ be a Rota-Baxter symplectic Leibniz algebra of
weight $\lambda$. By Theorem \ref{Sq}, $(A,\succ,\prec,\omega)$ is a quadratic Leibniz-dendriform algebra.
Using Eqs.~(\ref{C1}) and (\ref{Fs1}), for all $x,y,z\in A$, we have
\begin{align*}& \omega( P(x)\succ P(y)-P(P(x)\succ y+x\succ P(y)+\lambda x\succ y),z)
\\=&-\omega( P(x)\circ z, P(y))+\omega(P(x)\succ y, P(z))+\lambda\omega(P(x)\succ y, z)
+\omega(x\succ P(y), P(z))\\&+\lambda\omega(x\succ P(y), z)
+\lambda\omega(x\succ y, P(z))+\lambda^{2}\omega(x\succ y, z)
\\=&-\omega( P(x)\circ z, P(y))-\omega( P(x)\circ  P(z), y)-\lambda \omega( P(x)\circ z, y)
-\omega( x\circ  P(z), P(y))\\&-\lambda \omega( x\circ z, P(y))
-\lambda \omega( x\circ P(z), y)
-\lambda^{2} \omega( x\circ z, y)
\\=&-\omega( P(x)\circ z+x\circ  P(z)+\lambda x\circ z, P(y))-\omega( P(x)\circ  P(z), y)
-\lambda \omega( P(x)\circ z+x\circ P(z)+\lambda x\circ z, y)
\\=&0,\end{align*}
which implies that $P(x)\succ P(y)=P(P(x)\succ y+x\succ P(y)+\lambda x\succ y)$.
Analogously, $P(x)\prec P(y)=P(P(x)\prec y+x\prec P(y)+\lambda x\prec y)$.
Thus, $(A,\succ,\prec,P,\omega)$ is a quadratic Rota-Baxter Leibniz-dendriform algebra of weight $\lambda$.
The other hand is apparently.
We complete the proof.
\end{proof}

Let $\omega$ be a non-degenerate bilinear form on a vector space $A$. Then there is an isomorphism
$\omega^{\sharp}:A\longrightarrow A^{*}$ given by
\begin{equation} \omega(x,y)=\langle\omega^{\sharp}(x),y \rangle,~~\forall~x,y\in A.\end{equation}
Define an element $r_{\omega}\in A\otimes A$ such that $T_{r_{\omega}}=(\omega^{\sharp})^{-1}$,
that is,
\begin{equation}\label{Nd1} \langle T_{r_{\omega}}(\zeta),\eta\rangle=\langle r_{\omega},\zeta \otimes\eta\rangle=
\langle (\omega^{\sharp})^{-1}(\zeta), \eta\rangle, \ \ \forall~\zeta,\eta\in A^{*}. \end{equation}

\begin{lem}\label{Fb1} Let $(A, \succ,\prec)$ be a Leibniz-dendriform algebra and $\omega$ be a non-degenerate bilinear form on $A$. Then
$(A, \succ,\prec,\omega)$ is a quadratic Leibniz-dendriform algebra if and only if the corresponding $r_{\omega}\in A\otimes A$ given by
Eq.~(\ref{Nd1}) is symmetric and invariant.\end{lem}
\begin{proof}
It is obvious that $\omega$ is symmetric if and only if $r_{\omega}$ is symmetric.
For all $x,y\in A$, put $\omega^{\sharp}(x)=\zeta,\omega^{\sharp}(y)=\eta,
\omega^{\sharp}(z)=\theta$ with $\zeta,\eta,\theta\in A^{*}$. If $\omega$ is invariant, we have
\begin{align*} \omega (x \succ y, z)+\omega(y, x\circ z)
=&\langle \theta, x\succ(\omega^{\sharp})^{-1}(\eta)\rangle
+\langle \eta, x\circ (\omega^{\sharp})^{-1}(\theta)\rangle
\\=&\langle \theta, x\succ T_{r_{\omega}}(\eta)\rangle
+\langle \eta, x\circ T_{r_{\omega}}(\theta)\rangle
\\=&\langle -R_{\succ}^{*}(T_{r_{\omega}}(\eta))\theta-R_{\circ}^{*}(T_{r_{\omega}}(\theta))\eta,  x\rangle\\=&0,\end{align*}
\begin{align*} \omega(y\prec x, z)-\omega(x\circ z+z\circ x, y)
=&\langle \omega^{\sharp}(z),(\omega^{\sharp})^{-1}(\eta)\prec x\rangle-
\langle \omega^{\sharp}(y),x\circ(\omega^{\sharp})^{-1}(\theta)+(\omega^{\sharp})^{-1}(\theta)\circ x
\rangle
\\=&\langle \theta, T_{r_{\omega}}(\eta)\prec x\rangle-
\langle \eta,x\circ T_{r_{\omega}}(\theta)+T_{r_{\omega}}(\theta)\circ x
\rangle
\\=&\langle L_{\star}^{*}(T_{r_{\omega}}(\theta))\eta-L_{\prec}^{*}(T_{r_{\omega}}(\eta))\theta,x\rangle\\=&0.
\end{align*}
Thus,
\begin{align}&\label{Sd1}R_{\succ}^{*}(T_{r_{\omega}}(\eta))\theta=R_{\circ}^{*}(T_{r_{\omega}}(\theta))\eta,
\ \ \ L_{\star}^{*}(T_{r_{\omega}}(\theta))\eta=L_{\prec}^{*}(T_{r_{\omega}}(\eta))\theta.\end{align}
It follows that
\begin{align}\label{Sd2}
R_{\odot}^{*}(T_{r_{\omega}}(\eta))\theta=L_{\circ}^{*}(T_{r_{\omega}}(\theta))\eta.\end{align}
Combining Proposition \ref{In3} and Eqs.~(\ref{Sd1})-(\ref{Sd2}), $r_{\omega}\in A\otimes A$ is invariant.
The converse part can be checked similarly.
 \end{proof}

\begin{pro} \label{QF1} Let $(A, \succ,\prec,\omega)$ be a quadratic Leibniz-dendriform algebra and $r\in A\otimes A$.
Assume that $r+\tau(r)$ is invariant. Define a linear map
\begin{equation}\label{Nd2}P:A\longrightarrow A,\ \ \ P(x)=T_{r}\omega^{\sharp}(x),~~\forall~x\in A.\end{equation}
Then $r$ is a solution of the LD-YBE in $(A, \succ,\prec)$ if and only if $P$ satisfies
\begin{align}&\label{Nd3}P(x)\succ P(y)= P(P(x)\succ y+x\succ P(y)-x\succ T_{r+\tau(r)}\omega^{\sharp}(y)),
\\&\label{Nd4}P(x)\prec P(y)= P(P(x)\prec y+x\prec P(y)-x\prec T_{r+\tau(r)}\omega^{\sharp}(y)),~~\forall~x,y\in A.
 \end{align}
\end{pro}

\begin{proof} By Lemma \ref{Fb1}, $r_{\omega}$ is symmetric and invariant. Using
 Eqs.~(\ref{Iv9}), (\ref{AD1}) and (\ref{Nd1}),
 for all $x,y\in A$, put $\omega^{\sharp}(x)=\zeta,\omega^{\sharp}(y)=\eta$ with
$\zeta,\eta\in A^{*}$, we get
\begin{align*}& P(x)\succ P(y)=T_{r}\omega^{\sharp}(x)\succ T_{r}\omega^{\sharp}(y)=T_{r}(\zeta)\succ T_{r}(\eta),\\
&P(P(x)\succ y)=T_{r}\omega^{\sharp}(T_{r}(\zeta)\succ T_{r_{\omega}}(\eta))=
T_{r}\omega^{\sharp} T_{r_{\omega}}(L_{\circ}^{*}(T_{r}(\zeta))\eta)
=T_{r}(L_{\circ}^{*}(T_{r}(\zeta))\eta),
\\
&P(x\succ P(y))=T_{r}\omega^{\sharp}(T_{r_{\omega}}(\zeta)\succ T_{r}(\eta))=
T_{r}\omega^{\sharp}T_{r_{\omega}}(R_{\odot}^{*}(T_{r}(\eta))\zeta)
=T_{r}(R_{\odot}^{*}(T_{r}(\eta))\zeta),
\\
&P(x\succ T_{r+\tau(r)}\omega^{\sharp}(y))
=T_{r}\omega^{\sharp}(T_{r_{\omega}}(\zeta)\succ T_{r+\tau(r)}(\eta))
=T_{r}\omega^{\sharp}  T_{r_{\omega}}( R_{\odot}^{*}(T_{r+\tau(r)}(\eta))\zeta)
\\=&-T_{r}(( R_{\odot}^{*}(T_{r+\tau(r)}(\eta))\zeta)
=T_{r}(\zeta\succeq_{r+\tau(r)}\eta).\end{align*}
 Thus, Eq.~(\ref{Nd3}) holds if and only if Eq.~(\ref{AD6}) holds.
 Analogously, Eq.~(\ref{Nd4}) holds if and only if Eq.~(\ref{AD5}) holds.
 Combining Theorem \ref{Al}, we get the statement.
 \end{proof}

\begin{lem} \label{QF2} Let $A$ be a vector space and $\omega$ be a non-degenerate symmetric bilinear form.
Let $r\in A\otimes A$,
$\lambda\in k$ and $P$ be given by Eq.~(\ref{Nd2}). Then $r$ satisfies
\begin{equation} \label{Nd5} r+\tau(r)=-\lambda r_{\omega} \end{equation}
if and only if $P$ satisfies Eq.~(\ref{Fs}).
\end{lem}
\begin{proof} For all $x,y\in A$, put $\omega^{\sharp}(x)=\zeta,\omega^{\sharp}(y)=\eta,~~\zeta,\eta\in A^{*}$.
\begin{align*}& \omega(P(x),y)=\omega(y,P(x))=\langle\omega^{\sharp}(y),T_{r}\omega^{\sharp}(x)\rangle
=\langle \eta,T_{r}(\zeta)\rangle
=\langle r,\zeta\otimes\eta\rangle,\\
&\omega(x,P(y))=\langle \omega^{\sharp}(x),T_{r}\omega^{\sharp}(y)\rangle
=\langle \zeta,T_{r}(\eta)\rangle
=\langle r,\eta\otimes\zeta\rangle
=\langle \tau(r),\zeta\otimes\eta\rangle,
\\
&\lambda\omega(x,y)=\lambda \omega(y,x)
=\lambda\langle \omega^{\sharp}(y),(\omega^{\sharp})^{-1}\omega^{\sharp}(x)\rangle
=\lambda\langle r_{\omega},\zeta\otimes\eta\rangle
.\end{align*}
Thus, Eq.~(\ref{Nd5}) holds if and only if Eq.~(\ref{Fs}) holds.
 \end{proof}

\begin{cor} \label{Fb3} Let $(A, \succ,\prec,P,\omega)$ be a quadratic Rota-Baxter Leibniz-dendriform algebra of weight 0.
 Then there is a triangular
Leibniz-dendriform bialgebra $(A, \succ,\prec, \Delta_{\succ,r},\Delta_{\prec,r})$ with $\Delta_{\succ,r},\Delta_{\prec,r}$ given
 by Eqs.~(\ref{CB1})-(\ref{CB2}), where $r\in A\otimes A$ given by
$T_{r}(\zeta)=P(\omega^{\sharp})^{-1}(\zeta)$ for all $\zeta\in A^{*}$.
\end{cor}

 \begin{proof} Since $r + \tau(r) = -\lambda r_{\omega} = 0$, it follows that $r$ is skew-symmetric.
In view of Proposition \ref{QF1} and Lemma \ref{QF2}, the desired conclusion follows.
 \end{proof}

\begin{thm} \label{Fb3} Let $(A, \succ,\prec, \Delta_{\succ,r},\Delta_{\prec,r})$ be a factorizable
 Leibniz-dendriform bialgebra
with $r\in A\otimes A$. Then $(A, \succ,\prec,P,\omega)$
 is a quadratic Rota-Baxter Leibniz-dendriform algebra of weight $\lambda$ with $P$ given by Eq. (\ref{Nd2}), and $\omega$ is given by
\begin{equation} \label{Nd6} \omega(x,y)=-\lambda\langle T_{r+\tau(r)}^{-1}(x),y \rangle,~~\forall~x,y\in A.\end{equation}
Conversely, let $(A, \succ,\prec,P,\omega)$ be a quadratic Rota-Baxter Leibniz-dendriform algebra of weight $\lambda~(\lambda\neq 0)$.
Then there is a factorizable Leibniz-dendriform bialgebra $(A, \succ,\prec,\Delta_{\succ,r},\Delta_{\prec,r})$
 with $\Delta_{\succ,r},\Delta_{\prec,r}$ defined by Eqs.~(\ref{CB1})-(\ref{CB2}), where $r\in A\otimes A$ is
given through the operator form $T_r=P(\omega^{\sharp})^{-1}$.
\end{thm}

\begin{proof} On the one hand, since $(A, \succ,\prec, \Delta_{\succ,r},\Delta_{\prec,r})$ is a factorizable
Leibniz-dendriform bialgebra, $r+\tau(r)$ is invariant and $T_{r+\tau(r)}$ is a linear isomorphism.
In view of Proposition \ref{QF1} and Lemma \ref{QF2}, we know that $(A, \succ,\prec,P,\omega)$
 is a quadratic Rota-Baxter Leibniz-dendriform algebra of weight $\lambda$, where $\omega^{\sharp}=-\lambda T_{r+\tau(r)}^{-1}$.
 Conversely, suppose that $(A, \succ,\prec,P,\omega)$ is a quadratic Rota-Baxter Leibniz-dendriform
 algebra of weight $\lambda~(\lambda\neq 0)$.
 By Lemma \ref{Fb1}, Lemma \ref{QF2} and Proposition \ref{QF1}, we know that
 $r+\tau(r)$ is invariant, $T_{r+\tau(r)}=-\lambda (\omega^{\sharp})^{-1}$ is a linear isomorphism
 and $r$ is a solution of the LD-YBE in $(A, \succ,\prec)$. Therefore,
  $(A, \succ,\prec,\Delta_{\succ,r},\Delta_{\prec,r})$ is a factorizable Leibniz-dendriform bialgebra.

\end{proof}

\begin{pro} Let $(A, \succ,\prec,P)$ be a Rota-Baxter Leibniz-dendriform algebra of weight $\lambda$. Then
$(A\ltimes A^{*},\omega,P-(P^{*}+\lambda I))$
is a quadratic Rota-Baxter Leibniz-dendriform algebra of weight $\lambda$, where the bilinear form $\omega$ on $A\oplus A^{*}$
is given by
\begin{equation*} \omega(x+\zeta,y+\eta)=\langle x,\eta\rangle+\langle y,\zeta\rangle,~~\forall~x,y\in A,~\zeta,\eta\in A^{*}.\end{equation*}
Then $(A\ltimes A^{*}, \Delta_{\succ,r},\Delta_{\prec,r})$
is a factorizable Leibniz-dendriform bialgebra
with $\Delta_{\succ,r},\Delta_{\prec,r}$ defined by Eqs.~(\ref{CB1})-(\ref{CB2}) with $r$ given by $T_r=P(\omega^{\sharp})^{-1}$.
  Explicitly, assume that $\{e_1,\cdot\cdot\cdot, e_n\}$
is a basis of $A$ and $\{e_{1}^{*},\cdot\cdot\cdot, e^{*}_n\}$
is the dual basis, where $r=\sum_{i}e_{i}^{*}\otimes P(e_{i})-(P+\lambda I)(e_{i})\otimes e_{i}^{*}$

\end{pro}

\begin{proof} By direct computations, $(A\ltimes A^{*},\omega,P-(P^{*}+\lambda I))$
is a quadratic Rota-Baxter Leibniz-dendriform algebra of weight $\lambda$.
For all $x\in A$ and $\zeta\in A^{*}$, $\omega^{\sharp}(x+\zeta)=x+\zeta$. By Corollary \ref{Fb3}, we have a linear map
$T_{r}:A\oplus A^{*}\longrightarrow A\oplus A^{*}$ defined by
\begin{equation*}
T_{r}(x+\zeta)=(P-(P^{*}+\lambda I))(\omega^{\sharp})^{-1}(x+\zeta)=P(x)-(P^{*}+\lambda I)(\zeta).\end{equation*}
Thus,
\begin{align*}&
\sum_{i,j}\langle r,e_{i}\otimes e_{j}^{*}\rangle=\sum_{i,j}\langle T_{r}(e_{i}), e_{j}^{*}\rangle=\sum_{i,j}\langle P(e_{i}), e_{j}^{*}\rangle
\\&\sum_{i,j}\langle r,e_{i}^{*}\otimes e_{j}\rangle=\sum_{i,j}\langle T_{r}(e_{i}^{*}), e_{j}\rangle=-\sum_{i,j}\langle (P^{*}+\lambda I)(e_{i}^{*}), e_{j}\rangle=-\sum_{i,j}\langle e_{i}^{*}, (P+\lambda I)(e_{j})\rangle.\end{align*}
It follows that $r=\sum_{i}e_{i}^{*}\otimes P(e_{i})-(P+\lambda I)(e_{i})\otimes e_{i}^{*}$.
\end{proof}

 By Theorem \ref{Fb3}, there exists a one-to-one correspondence between quadratic
  Rota-Baxter Leibniz-dendriform algebras of nonzero weight $\lambda$ and
   factorizable Leibniz-dendriform bialgebras. Combining Theorem \ref{Fb0},
   Rota-Baxter symplectic Leibniz algebras of nonzero weight also correspond to such bialgebras.

Based on Proposition~\ref{Fb2} and Theorem~\ref{Fb3}, we observe that a quadratic Rota-Baxter Leibniz-dendriform
 algebra $(A, \ast, \circ, P, \omega)$ of nonzero weight $\lambda$ corresponds to a
 factorizable Leibniz-dendriform bialgebra.
It follows that $(A, \ast, \circ, -\lambda I-P, \omega)$ also corresponds to such a bialgebra.
Moreover, by Proposition~\ref{Qf} and Theorem~\ref{Fb3}, if $(A, \ast, \circ, \Delta_r, \delta_r)$
is a factorizable Leibniz-dendriform bialgebra, then $(A, \ast, \circ, \Delta_{\tau(r)}, \delta_{\tau(r)})$
also gives rise to a quadratic Rota-Baxter Leibniz-dendriform algebra of weight $\lambda$.

\begin{pro}
Let $(A, \succ,\prec, \Delta_{\succ,r},\Delta_{\prec,r})$ be a factorizable
Leibniz-dendriform bialgebra which corresponds to a quadratic Rota-Baxter Leibniz-dendriform algebras of non-zero weight $\lambda$.
Then the factorizable
Leibniz-dendriform bialgebra $(A, \succ,\prec, \Delta_{\succ,\tau(r)},\Delta_{\prec,\tau(r)})$ corresponds to the
quadratic Rota-Baxter Leibniz-dendriform algebra $(A,\succ,\prec,-\lambda I-P,\omega)$ of non-zero weight $\lambda$.
\end{pro}

\begin{proof}
In view of Proposition \ref{Qf} and Theorem \ref{Fb1}, the factorizable
Leibniz-dendriform bialgebra $(A, \succ,\prec, \Delta_{\succ,\tau(r)},\Delta_{\prec,\tau(r)})$ corresponds to a
quadratic Rota-Baxter Leibniz-dendriform algebra $(A,$ \ \ \ $\succ,\prec,P',\omega')$ of non-zero weight $\lambda$.
By Theorem \ref{Fb3},
\begin{equation*}
\omega'(x,y)=-\lambda\langle T_{r+\tau(r)}^{-1}(x),y \rangle=\omega(x,y).\end{equation*}
Using Eqs.~(\ref{Nd2}) and (\ref{Nd6}), we obtain
\begin{align*}
P'(x)= T_{\tau(r)}{\omega^{'}}^{\sharp}(x)=T_{\tau(r)}\omega^{\sharp}(x)
&=-\lambda T_{\tau(r)}T_{r+\tau(r)}^{-1}(x)=\lambda (T_{r}-T_{r+\tau(r))})T_{r+\tau(r)}^{-1}(x)
\\&=\lambda T_{r}T_{r+\tau(r)}^{-1}(x)-\lambda (x)
=-T_{r}\omega^{\sharp}(x)-\lambda (x)=-P(x)-\lambda (x)
.\end{align*}
Therefore, the factorizable Leibniz-dendriform bialgebra
$(A, \succ, \prec, \Delta_{\succ,\tau(r)}, \Delta_{\prec,\tau(r)})$ yields a
quadratic Rota-Baxter Leibniz-dendriform algebra
$(A, \succ, \prec, -\lambda I-P, \omega)$ of nonzero weight $\lambda$.
A similar reasoning shows that the converse also holds.

\end{proof}

\begin{center}{\textbf{Acknowledgments}}
\end{center}
The first author is supported by the Natural Science
Foundation of Zhejiang Province of China (No. LY19A010001) and the Science
and Technology Planning Project of Zhejiang Province
(No. 2022C01118).  The second author is supported by the NSF of China (No. 12161013).

\begin{center} {\textbf{Statements and Declarations}}
\end{center}
 All datasets underlying the conclusions of the paper are available
to readers. No conflict of interest exits in the submission of this
manuscript.


\end {document}